\documentclass[11pt]{article}

\usepackage{vien}

\usepackage[utf8]{inputenc} % allow utf-8 input
\usepackage[T1]{fontenc}    % use 8-bit T1 fonts
\usepackage{mathtools,bm,epstopdf,url}
\usepackage{algorithm}    
\usepackage{enumitem}
\usepackage{tikz}
\usepackage{subcaption}
\usepackage{bm}
\usepackage{graphicx}
\usepackage{float}
\usepackage[final]{changes}
\usetikzlibrary{calc}
\usetikzlibrary{positioning}
\usetikzlibrary{shapes,arrows}
\usepackage[top=2.54cm,right=3cm,left=3cm,bottom=2.54cm]{geometry}

\title{Anderson Acceleration in Nonsmooth Problems: Local Convergence via Active Manifold Identification}

\author{	
	Kexin Li{\thanks{School of Information Science and Technology, ShanghaiTech University, Shanghai, People's Republic of China. Emails: \tt\small likx0403@gmail.com, bailw@shanghaitech.edu.cn}}
	\and
	Luwei Bai \footnotemark[1]
	\and 
	Xiao Wang {\thanks{Department of AI Computing, Peng Cheng Laboratory, Shenzhen, People's Republic of China. Emails: \tt\small wangx07@pcl.ac.cn}}  
	\and  
	Hao Wang {\thanks{Corresponding author. School of Information Science and Technology, ShanghaiTech University, Shanghai, People's Republic of China. Emails: \tt\small wanghao1@shanghaitech.edu.cn}}       
}

\begin{document}
\tikzstyle{block} = [rectangle, draw, fill=blue!20, 
text width=21em, text centered, rounded corners, minimum height=2em, align=flush left]
\tikzstyle{block1} = [rectangle, draw, fill=blue!20, 
text width=21em, text centered, rounded corners, minimum height=2em, align=flush left]
\tikzstyle{line} = [draw, -latex']
\tikzstyle{block2} = [rectangle, draw, fill=blue!20, 
text width=12em, text centered, rounded corners, minimum height=2.5em]
\tikzstyle{line} = [draw, -latex']
\maketitle
\begin{abstract}
	Anderson acceleration is an effective technique for enhancing the efficiency of fixed-point iterations; however, analyzing its convergence in nonsmooth settings presents significant challenges. In this paper, we investigate a class of nonsmooth optimization algorithms characterized by the active manifold identification property. This class includes a diverse array of methods such as the proximal point method, proximal gradient method, proximal linear method, proximal coordinate descent method, Douglas-Rachford splitting (or the alternating direction method of multipliers), and the iteratively reweighted $\ell_1$  method, among others. Under the assumption that the optimization problem possesses an active manifold at a stationary point, we establish a local R-linear convergence rate for the Anderson-accelerated algorithm. Our extensive numerical experiments further highlight the robust performance of the proposed Anderson-accelerated methods.
\end{abstract}

\section{Introduction}
In this paper, we focus on the following optimization problem:% % \vspace{-12pt}  % 手动减少公式上方间距
\begin{equation}\label{Nonsmooth optimization problem}
    \setlength{\abovedisplayskip}{8pt}   % 上方间距
    \setlength{\belowdisplayskip}{5pt}   % 下方间距
    \min\limits _{\bm{x}\in \mathbb{R}^n} F(\bm{x}),
\end{equation}
where, unless stated otherwise, $ F: \mathbb{R}^n \rightarrow \mathbb{R} \cup \{\infty\} $ is assumed to be a closed and level-bounded function, potentially nonsmooth and composed of both smooth and nonsmooth terms. Such problems commonly arise in various fields, including statistics, signal processing, and machine learning
%\cite{boyd2011distributed,peng2016coordinate}
. Depending on the specific characteristics of these problems, a diverse range of optimization algorithms has been developed. These include the proximal point algorithm, proximal gradient algorithm, proximal linear algorithm, and the Douglas-Rachford splitting algorithm, among others. A key feature of these algorithms is that their iterative schemes can often be reformulated as a fixed-point iteration for the following problem:
\begin{equation}\label{fixed-point iteration}
    \setlength{\abovedisplayskip}{8pt}   % 上方间距
    \setlength{\belowdisplayskip}{5pt}   % 下方间距
     \mbox{Find } \bm{x} \in \mathbb{R}^n \mbox{ such that } \bm{x} = H(\bm{x}),
\end{equation}
where $ H: \mathbb{R}^n \rightarrow \mathbb{R}^n $ is a continuous mapping. For convex optimization problems, the mapping $ H $ is typically nonexpansive, and the solution set of \eqref{fixed-point iteration} either coincides with or is closely related to the solution set of the original optimization problem.

Anderson acceleration is a widely recognized technique known for its effectiveness in enhancing the convergence of fixed-point problems, particularly in the later stages of iteration. By leveraging information from previous iterations, it combines weighted past iterations to generate a new one. Originally introduced to accelerate the iteration process of nonlinear integral equations, this method has since been extended to general fixed-point problems \cite{toth2015convergence}. Anderson acceleration has been extensively studied and applied, yielding promising results in practical computations.

Despite its notable efficiency and widespread application, the theoretical convergence analysis of Anderson acceleration lags significantly behind its practical use. Counterexamples indicate that global convergence may fail under relatively mild conditions. Many studies have focused on its local convergence and rate. However, proving local convergence in the absence of continuous differentiability is particularly challenging, even though Anderson acceleration does not require derivatives. Local R-linear convergence has been established for the linear mapping \cite{toth2015convergence}, the continuously differentiable mapping \cite{chen2019convergence}, and the Lipschitz continuously differentiable mapping \cite{toth2015convergence}.

For nonsmooth problem, contemporary research frequently incorporates strategies such as restarting and safeguarding steps to ensure convergence. Safeguarded Anderson-accelerated variants have been proposed for several algorithms, including the proximal gradient algorithm \cite{mai2020anderson}, coordinate descent algorithm \cite{bertrand2021anderson}, and Douglas-Rachford splitting algorithm \cite{fu2020anderson,ouyang2020anderson}. However, these globalized algorithms always fail to provide complexity guarantees. There is still a lack of theoretical analysis to explain the superior local performance of Anderson acceleration in nonsmooth optimization algorithms. To our knowledge, existing studies guarantee local linear convergence rates for Anderson acceleration algorithms only for certain specific nonsmooth problems. For example, Bian et al. \cite{bian2021anderson} study a fixed-point problem (with solution $\bm{x}^*$) decomposed into a smooth and a nonsmooth component with a small Lipschitz constant (where the constant approaches 0 as $\bm{x}\rightarrow \bm{x}^*$), demonstrating local linear convergence. Additionally, Bian and Chen \cite{bian2022anderson} consider a class of nonsmooth composite fixed-point problems $\bm{x} = H(\bm{x}):=(F\circ P_{\Omega}\circ Q)(\bm{x})$, where $F$ and $Q$ are Lipschitz continuously differentiable and $P_{\Omega}$ is a box set projection. They smooth $H$ and prove linear convergence for the Anderson-accelerated smoothing approximation. {Mai and Johansson \cite{mai2020anderson} propose an Anderson-accelerated proximal gradient (forward-backward splitting) algorithm with a local linear convergence rate guarantee. However, their method exclusively accelerates the forward step, the smooth part of the iteration—i.e., for a composite function \( F := f + g \) where \( f \) is smooth and \( g \) is nonsmooth. The Anderson acceleration is applied only to the forward step \( \bm{x}^{k+1/2} = \bm{x}^k - \nabla f(\bm{x}^k)/L \), while the backward step \( \bm{x}^{k+1} = \text{prox}_{\beta g}(\bm{x}^{k+1/2}) \), the nonsmooth part, remains unaccelerated.}

{\bf Contributions.} %Given the success of Anderson acceleration in expediting various nonsmooth optimization problems and the absence of a corresponding convergence rate analysis that theoretically substantiates it, this paper primarily contributes in the following ways.
In this paper, we present a general theoretical framework for analyzing the local convergence rates of Anderson acceleration applied to various nonsmooth optimization algorithms that possess the active manifold identification property. These algorithms include the proximal point algorithm (PPA), proximal gradient algorithm (PGA), proximal linear algorithm (PLA), proximal coordinate descent algorithm (PCD), Douglas-Rachford splitting algorithm (DRS), alternating direction method of multipliers (ADMM), and the iteratively reweighted $\ell_1$ algorithm (IRL1), among others, all of which tackle a wide range of nonsmooth optimization problems. We establish the local smoothness of iteration mappings for these algorithms, utilizing the corresponding theoretical framework. Additionally, we demonstrate local R-linear convergence rates for Anderson-accelerated versions of these algorithms. Figure \ref{Flow of theoretical analysis} illustrates the structure of our theoretical analysis, while Table \ref{Tab: comparison with existing work} summarizes the current progress and convergence results for the listed algorithms. Finally, we validate our theoretical findings and showcase the strong performance of Anderson-accelerated algorithms in our experiments.

\begin{table}[htbp]
    % \vspace{-0pt}
    \renewcommand\arraystretch{1}
    %\scalebox{1.5}{
      \begin{center}
    \begin{tabular}{@{}llllll@{}}
    %\begin{tabular}{@{}lll@{}}
    \hline
      & PPA/PLA  & PGA & PCD & DRS /ADMM & IRL1   \\
    \hline
     Global convergence  & -  &  \cite{mai2020anderson} &  \cite{bertrand2021anderson} &  \cite{fu2020anderson,ouyang2020anderson} & -\\
     Local convergence rate & Ours  & Ours,  \cite{mai2020anderson} & Ours & Ours & Ours  \\
    \hline
    \end{tabular}
  \end{center}
  % % \vspace{-7pt}
    \caption{Comparison with existing  Anderson accelerations in nonsmooth settings}\label{Tab: comparison with existing work}
\end{table}

\begin{figure}[htbp]
	% \vspace{-10pt}
	  \centering
	  % \hspace*{-27pt}
	  % % \vspace{0.2cm}
	% \begin{tikzpicture}[auto, transform canvas={scale=0.8}]
	\begin{tikzpicture}[auto]
		% nodes
		\node [block1] (init) { \scriptsize   {\bf Condition I}:  $F$ 
		admits an active manifold $\mathcal{M}$ at $\bm{x}^*$.};
		\node [block, below=0.15cm of init] (contact) {\scriptsize  {\bf Condition II}: A nonsmooth optimization algorithm identifies the active manifold: $H(\bm{x})\in \mathcal{M} \mbox{ for all } \bm{x} \mbox{ near } \bm{x}^* $.};
		\node [block2, right=5.2cm of  $(init)!0.5!(contact)$] (smooth) {\setlength{\baselineskip}{15em}\scriptsize   Local smoothness of $H$; \\ Local linear convergence rate of Anderson-accelerated algorithm.};
		% edges
		\coordinate (mid) at ($(smooth.west)+(-0.6cm,0)$);
		\draw [-, line width=1pt] (init.east) -| (mid);
		\draw [-, line width=1pt] (contact.east) -| (mid);
		\draw [->, >=stealth, line width=1pt] (mid) -- (smooth);
	\end{tikzpicture}
	% % \vspace{-15pt}
	  \caption{Flow of theoretical analysis to Anderson-accelerated nonsmooth optimization algorithms for \eqref{Nonsmooth optimization problem} with $\bm x^*$ being a Clarke critical point. % {I added an option on lines 153-155 to adjust the arrow size. The size can be scaled by adjusting the “line width”.}
	  } \label{Flow of theoretical analysis}
	  % \vspace{-13pt}
	\end{figure}
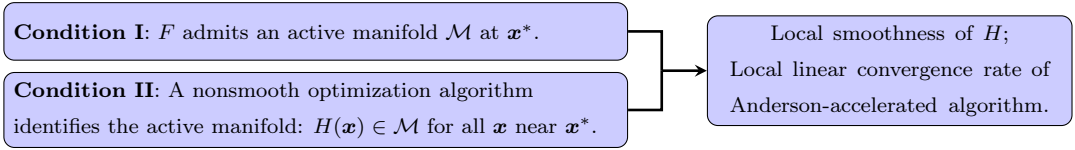

\section{Notation and preliminaries}
	% % \vspace{-8pt}
	We denote $\mathbb{N} := \{0,1,2,...\}$ and $\mathbb{R}_+^n$ as the nonnegative orthant in $\mathbb{R}^n$, with $\mathbb{R}_{++}^n$ representing its interior. The norm $\|\bm{x}\|_p = (\sum_{i=1}^{n}|x_i|^p)^{1/p}$ is defined for $p \in (0, +\infty)$, and unless otherwise specified, $\|\cdot\|$ refers to the $\ell_2$ norm. The function ${\rm sign}(\bm{x})$ returns the element-wise sign of $\bm{x}$, i.e., ${\rm sign}(\bm{x}) = [{\rm sign}(x_1), ..., {\rm sign}(x_n)]^T$, and ${\rm diag}(\bm{x}) \in \mathbb{R}^{n \times n}$ is the diagonal matrix with the elements of $\bm{x}$ on its diagonal. {Let \( \mathcal{B}_{\bm{x}}(\rho) \) and \( \hat{\mathcal{B}}_{\bm{x}}(\rho) \) denote the closed and open neighborhoods of \( \bm{x} \) with radius \( \rho > 0 \).} We use $\nabla f(\bm{x})$ to denote the gradient of $f$ at $\bm{x}$, and $\nabla_i f(\bm{x})$ to denote the partial derivative with respect to $x_i$, i.e., $\frac{\partial f(\bm{x})}{\partial x_i}$. A function $f:\mathbb{R}^n \rightarrow \mathbb{R} \cup \{\infty\}$ is called $\sigma$-weakly convex if the quadratically perturbed function $\bm{x} \mapsto f(\bm{x}) + \frac{\sigma}{2}\|\bm{x}\|^2$ is convex. Its {\it epigraph} is defined as $\mbox{epi} f := \{(\bm{x}, r) \in \mathbb{R}^n \times \mathbb{R} \mid r \ge f(\bm{x})\}$. For a set $\mathcal{M} \subseteq \mathbb{R}^n$, the indicator function $\chi_{\mathcal{M}}$ takes the value 0 in $\mathcal{M}$ and $+\infty$ otherwise. A mapping $H: \mathbb{R}^n \rightarrow \mathbb{R}^n$ is called a {\it contraction} if there exists a constant $\gamma \in (0,1)$ such that $\|H(\bm{x}) - H(\bm{y})\| \le \gamma \|\bm{x} - \bm{y}\|$ for all $\bm{x}, \bm{y} \in \mathbb{R}^n$. The Jacobian matrix of a mapping $G: \mathbb{R}^n \rightarrow \mathbb{R}^m$ is denoted by $\nabla G$, and $\mathbb{I}_d$ represents the identity mapping from $\mathbb{R}^d$ to $\mathbb{R}^d$. For a set $S$, the {\it relative interior} ${\rm rint}(S)$ is defined as:
	$$    
		\setlength{\abovedisplayskip}{8pt}   % 上方间距
		\setlength{\belowdisplayskip}{5pt}   % 下方间距
		{\rm rint}(S) := \{\bm{x} \in S : \text{there exists } \epsilon > 0 \text{ such that } \hat{\mathcal{B}}_{\bm{x}}(\epsilon) \cap {\rm aff}(S) \subseteq S \},
	$$
	where ${\rm aff}(S)$ is the affine hull of $S$. Given a function $f: \mathbb{R}^n \rightarrow \mathbb{R} \cup \{\infty\}$ and a matrix $C \in \mathbb{R}^{m \times n}$, the {\it image function} $(Cf): \mathbb{R}^m \rightarrow \mathbb{R} \cup \{\infty\}$ is defined as:
	$(Cf)(\bm{s}) := \inf_{\bm{x} \in \mathbb{R}^n} \{f(\bm{x}) \mid C\bm{x} = \bm{s}\}.$
	\subsection{Anderson acceleration}
	% % \vspace{-5pt}
	Consider the fixed-point problem \eqref{fixed-point iteration}, and let ${\bm{x}^*}$ be its solution, i.e., a fixed point of $H$. The framework for Anderson-accelerated fixed-point iterations is described in Algorithm \ref{AA fixed-point}.
	
	  \begin{algorithm}
		\caption{Anderson acceleration %accelerated fixed-point iteration
		}\label{AA fixed-point}
		\begin{algorithmic}[1]
		\STATE{Given $ \bm{x}^0 \in \mathbb{R}^n$ and integer $m \ge 1$.}
		\STATE{Set $k = 0$, $H^0 = H(\bm{x}^0)$, and $\bm{r}^0 = H^0 - \bm{x}^0$.}
		\WHILE{not convergent}
		 \STATE{Set $m_k = \min(m, k)$ and $R^k = [\bm{r}^k, ..., \bm{r}^{k-m_k}]$}
		 \STATE{Compute %$\bm{\alpha}^k \leftarrow \arg\min_{\bm{\alpha}^T\textbf{e}=1}\|R^k\bm{\alpha}\|^2$
		 $
			\bm{\alpha}^k = \arg\min_{\bm{\alpha}^T\textbf{e}=1}\|R^k\bm{\alpha}\|^2
	$
	}
		 \STATE{Compute $\bm{x}^{k+1} = \sum^{m_k}_{i=0} \alpha^k_i H^{k-m_k+i}$}
		 \STATE{Compute $H^{k+1} = H(\bm{x}^{k+1})$ and $\bm{r}^{k+1} = H^{k+1} - \bm{x}^{k+1}$}
		 \STATE{Set $k = k+1$}
		\ENDWHILE
		\RETURN {$\bm{x}^{k}$} 
		\end{algorithmic}
		\end{algorithm}

In Algorithm \ref{AA fixed-point}, we define $H^k = H(\bm{x}^k)$, $\bm{r}^k = H(\bm{x}^k) - \bm{x}^k$ as the residue, and $R^k = [\bm{r}^k, ..., \bm{r}^{k-m_k}]$ for $k \geq 0$. Starting with an initial point $\bm{x}^0$ and an integer $m \geq 1$, the key step in Anderson acceleration (Step 5) involves calculating the weight vector $\bm{\alpha}^k \in \mathbb{R}^{m_k+1}$ that minimizes the weighted sum of the previous $m_k+1$ residues.
%\begin{equation}
%    \begin{aligned}\label{AA subproblem}
 %       \bm{\alpha}^k = \arg\min_{\bm{\alpha}^T\textbf{e}=1}\|R^k\bm{\alpha}\|^2.
%\end{aligned}
%\end{equation}
%The next iterate $\bm{x}^{k+1}$ is then obtained in Step 6 by using the weight vector $\bm{\alpha}^k$ to compute the weighted sum of $H^{k},...,H^{k-m_k}$. 
%\begin{equation}
%    \begin{aligned}
 %       \bm{x}^{k+1} = \sum^{m_k}_{i=0}\alpha^k_i H^{k-m_k+i}.
%\end{aligned}
%\end{equation}
It is straightforward to see that $\bm{\alpha}^k$ can be explicitly expressed as \cite{mai2020anderson}: $\bm{\alpha}^k = \frac{[(R^k)^T R^k]^{-1} \boldsymbol{1}}{\boldsymbol{1}^T [(R^k)^T R^k]^{-1} \boldsymbol{1}}.$ The computational cost of Step 5 is $O(m_k^2 + nm_k)$. Given that $m_k$ is typically a small integer in practice, such as $5-15$, solving this subproblem incurs minimal computation. Therefore, Anderson acceleration does not significantly increase the computational burden. Furthermore, if $(R^k)^T R^k$ is singular, we can ensure the subproblem’s non-singularity by adding a Tikhonov regularization term {$\tau \|\bm{\alpha}\|^2$} to the objective function \cite{mai2020anderson}, where $\tau > 0$ is a small constant.

In this paper, we make the following assumption regarding the fixed-point mapping $H$ and the coefficients generated by Anderson acceleration.

\begin{assumption} \label{Assumption AA} 
	Let $\bm{x}^*$ be the solution of the fixed-point problem $\bm{x} = H(\bm{x})$ and let $\{\bm{\alpha}^k\}$ denote the sequence generated by Anderson acceleration. 
	\begin{itemize}
	\item[(i)] There exist constants $\rho, \gamma > 0$ such that 
	$\|H(\bm{x}) - H(\bm{y})\| \leq \gamma \|\bm{x} - \bm{y}\|$ for all $\bm{x}, \bm{y} \in \mathcal{B}_{\bm{x}^*}(\rho)$
	\item[(ii)] There exists an upper bound $M_{\bm{\alpha}}$ such that $\sum_{i=0}^{m_k} |\alpha^k_i| \leq M_{\bm{\alpha}}$ for all $k \in \mathbb{N}$.
	\end{itemize}
\end{assumption}

These assumptions are commonly used in the literature on Anderson acceleration \cite{bian2022anderson,bian2021anderson,toth2015convergence}. Assumption \ref{Assumption AA}(ii) ensures the boundedness of $\sum_{i=0}^{m_k} |\alpha^k_i|$. In our experiments, we also observed this phenomenon. Although proving the boundedness remains elusive, several practical approaches have been proposed to ensure it \cite{scieur2016regularized,toth2015convergence}. 

The first convergence result for Anderson acceleration is presented in \cite[Theorem 2.3]{toth2015convergence}. Under the assumption of Lipschitz continuous differentiability of \( H \), Anderson acceleration guarantees R-linear convergence when initialized near the fixed point \( \bm{x}^* \).
    
% The theoretical analysis of  \cite[Theorem 2.3]{toth2015convergence} requires the continuous differentiability of $H$. However, in the context of problem \eqref{equation1}, the mapping defined in (\ref{H}) struggles to meet this requirement. Notably, Wang et al.  \cite{wang2021relating} suggested that IRL1 with $\ell_p$ norm is equivalent to gradient descent at the tail end of the algorithm in a smooth subspace. Inspired by this, we establish a local continuous differentiability of $H$.

\begin{theorem} \label{Convergence AA} 
Under Assumption \ref{Assumption AA}, suppose that $H$ is Lipschitz continuously differentiable in a neighborhood $\mathcal{B}_{\bm{x}^*}(\hat{\rho})$ of $\bm{x}^*$ for some $0 < \hat{\rho} \le \rho$ and $\bm{x}^0 \in \mathcal{B}_{\bm{x}^*}(\hat{\rho})$. Then, when $\bm{x}^0$ is sufficiently close to $\bm{x}^*$, the iterates $\{\bm{x}^k\}$ generated by Anderson acceleration remain in $\mathcal{B}_{\bm{x}^*}(\hat{\rho})$ and converge to $\bm{x}^*$ R-linearly with $\hat{\gamma} \in (\gamma, 1)$:\begin{equation}\label{cov-rate}    
    \setlength{\abovedisplayskip}{8pt}   % 上方间距
    \setlength{\belowdisplayskip}{5pt}   % 下方间距
        \|H^k - \bm{x}^k\| \leq \hat{\gamma}^k \|H^0 - \bm{x}^0\| \quad \text{and} \quad 
        \|\bm{x}^k - \bm{x}^*\| \leq \frac{1 + \gamma}{1 - \gamma} \hat{\gamma}^k \|\bm{x}^0 - \bm{x}^*\|.
\end{equation}
\end{theorem}

\subsection{Subdifferentials}
% \vspace{-8pt}
The following concepts about subdifferentials are well-established in the literature, such as \cite{davis2021subgradient}. The regular (Fréchet) normal cone to a set $Q \subset \mathbb{R}^n$ at a point $\bm{x} \in \mathbb{R}^n$ is defined as follows:
$$
    \setlength\abovedisplayskip{1pt} 
    \setlength\belowdisplayskip{1pt}
    N_Q(\bm{x}) := \left\{\bm{\nu} \in \mathbb{R}^n : \mathop{\lim\sup}_{\bm{y} \rightarrow \bm{x}, \bm{y} \in Q}  \frac{\langle \bm{\nu}, \bm{y} - \bm{x} \rangle}{\|\bm{y} - \bm{x}\|} \le 0 \right\}.
$$
The limiting normal cone to $Q$ at $\bm{x} \in Q$, denoted by $\hat{N}_Q(\bm{x})$, is defined to consist of all vectors $\bm{\nu} \in \mathbb{R}^n$ for which there exist sequences $\bm{x}_i \in Q$ and $\bm{\nu}_i \in \hat{N}_Q(\bm{x}_i)$ satisfying $(\bm{x}_i, \bm{\nu}_i) \rightarrow (\bm{x}, \bm{\nu})$. The Clarke normal cone, denoted by $N^c_Q(\bm{x})$, is defined as the closed convex hull of $\hat{N}_Q(\bm{x})$. For all points $\bm{x} \in Q$, it holds that $N_Q(\bm{x}) \subset \hat{N}_Q(\bm{x}) \subset N^c_Q(\bm{x})$. Generalized gradients of a function can be defined through the normal cones to epigraphs. Specifically, consider a function $f: \mathbb{R}^n \rightarrow \mathbb{R} \cup \{\infty\}$ and a point $\bm{x}$ such that $f(\bm{x})$ is finite. The subderivative $df(\bm{x})(\cdot)$  is defined by
$$
    \setlength{\abovedisplayskip}{8pt}   % 上方间距
    \setlength{\belowdisplayskip}{5pt}   % 下方间距
    df(\bm{x})(\bm{\overline{\nu}}) = \lim_{\tau \downarrow 0} \inf_{\bm{\nu} \rightarrow \bm{\overline{\nu}}} \frac{f(\bm{x} + \tau \bm{\nu}) - f(\bm{x})}{\tau}, \quad \bm{\nu} \in \mathbb{R}^n.
$$
The regular and Clarke subdifferentials of $f$ at $\bm{x}$ are respectively defined as
$$
    \setlength{\abovedisplayskip}{8pt}   % 上方间距
    \setlength{\belowdisplayskip}{5pt}   % 下方间距
    \begin{aligned}
    \partial f(\bm{x}) &:= \{\bm{\nu} \in \mathbb{R}^n : (\bm{\nu}, -1) \in N_{\mbox{epi} f}(\bm{x}, f(\bm{x}))\}, \\
    \partial_c f(\bm{x}) &:= \{\bm{\nu} \in \mathbb{R}^n : (\bm{\nu}, -1) \in N^c_{\mbox{epi} f}(\bm{x}, f(\bm{x}))\}.
\end{aligned}
$$
A point $\overline{\bm{x}}$ satisfying $\bm{0} \in \partial f(\bm{x})$ is called a {\it critical point} of $f$, and a point satisfying $\bm{0} \in \partial_c f(\bm{x})$ is called a {\it Clarke critical point}. Notably, the latter requirement is much weaker than the former. The distinction disappears for weakly convex functions \cite{davis2021subgradient}. We say that $f$ is (subdifferentially) regular at $\bm{x}$ if $\mbox{epi} f$ is locally closed around $\bm{x}$ and $N_Q(\bm{x}) = N^c_Q(\bm{x})$.

\subsection{Weak convexity and Moreau envelope}
For a $\sigma$-weakly convex function $f: \mathbb{R}^n \rightarrow \mathbb{R} \cup \{\infty\}$, it naturally has a lower bound:
$$
    \setlength{\abovedisplayskip}{8pt}   % 上方间距
    \setlength{\belowdisplayskip}{5pt}   % 下方间距
    f(\bm{y}) \ge f(\bm{x}) + \langle \bm{\nu}, \bm{y} - \bm{x} \rangle - \frac{\sigma}{2} \|\bm{y} - \bm{x}\|^2, \quad \forall \bm{x}, \bm{y} \in \mathbb{R}^n, \bm{\nu} \in \partial f(\bm{x}).
$$
The Moreau envelope and proximal point mapping of $f$ are defined as follows:
$$
    \setlength{\abovedisplayskip}{8pt}   % 上方间距
    \setlength{\belowdisplayskip}{5pt}   % 下方间距
    \begin{aligned}
            f_{\beta}(\bm{x}) &= \inf_{\bm{y} \in \mathbb{R}^n} \left\{ f(\bm{y}) + \frac{1}{2\beta} \|\bm{y} - \bm{x}\|^2 \right\}, \\
    {\rm prox}_{\beta f}(\bm{x}) &= \mathop{\arg\min}_{\bm{y} \in \mathbb{R}^n} \left\{ f(\bm{y}) + \frac{1}{2\beta} \|\bm{y} - \bm{x}\|^2 \right\}.
    \end{aligned}
$$
In the following lemma, we summarize the basic properties that are used in this paper. More properties can be found in \cite[Lemma 2.5]{davis2022proximal}.

%\red{More properties should be cited. For instance, it has shown the Lipschitz continuity of $\nabla f_\beta$. Can be applied to prove the Lipschitz continuity of $\nabla H$ in next sections? In  \cite[Theorem 3.1]{davis2022proximal}, Moreau Envelope $f_\beta$ is $C^2$-smooth. Is it helpful to derive the $C^3$-smoothness? If so, $\nabla^3 f_\beta$ is bounded near $x^*$, then by $\nabla f_\beta(x) = \beta^{-1}(x-\mbox{prox}_{\beta f}(x))$ we can obtain $\nabla^2 \mbox{prox}_{\beta f}(x)$ is bounded, then $\nabla H$ is Lipschitz continuously differentiable.}
\begin{lemma}
\label{Moreau envelope and the proximal point mapping}
    Let $f : \mathbb{R}^d \to \mathbb{R} \cup \{\infty\}$ be a $\sigma$-weakly convex function, and fix a positive parameter $\beta < \sigma^{-1}$. The following statements hold:
    \begin{enumerate}
        \item [{\rm (i)}] The Moreau envelope $f_\beta$ is $\mathcal{C}^1$-smooth and $\frac{\sigma}{1 - \beta\sigma}$-weakly convex.
        \item [{\rm (ii)}] The proximal mapping $\mathrm{prox}_{\beta f} $ is $\frac{1}{1 - \beta\sigma}$-Lipschitz continuous.
        \item [{\rm (iii)}] The critical points of $f$ and $f_\beta$ coincide, and they are exactly the fixed points of the proximal mapping $\mathrm{prox}_{\beta f}$.
    \end{enumerate}
\end{lemma}

\subsection{Manifold optimization}
Nonsmooth optimization problems often exhibit significant structure, with many studies demonstrating that the nonsmooth behavior of their objective functions is closely related to an active manifold in the local region \cite{hare2007identifying,lewis2002active}. Specifically, the critical points of typical nonsmooth functions lie on a manifold, where the functions are smooth. This relationship allows us to explore the local smoothness properties of nonsmooth optimization problems through the lens of manifold optimization. The concept of an active manifold is closely tied to the active set. From an algorithmic perspective, active manifolds are sets that typical algorithms can identify in finite time. The notion of active manifolds has been modeled in various ways, such as partly smooth manifolds \cite{lewis2002active}, $\mathcal{UV}$-structures \cite{mifflin2005algorithm}, $g \circ F$ decomposable functions \cite{shapiro2003class}, and minimal identifiable sets \cite{drusvyatskiy2014optimality}. We first revisit the definition of a smooth manifold \cite[Definition 2.2]{davis2022proximal}.

\begin{definition}[Smooth manifold]\label{def:2.2}
    A subset $\mathcal{M} \subset \mathbb{R}^n$ is called a $\mathcal{C}^p$-smooth manifold (for $p \geq 1$) of dimension $r$ around $\overline{\bm{x}} \in \mathcal{M}$ if there exists an open neighborhood $\hat{\mathcal{B}}_{\overline{\bm{x}}}(\epsilon)$ around $\overline{\bm{x}}$ with radius $\epsilon>0$ and a mapping $G: \mathbb{R}^n \to \mathbb{R}^{n-r}$ such that:
    \begin{enumerate}
        \item [{\rm (i)}] $G$ is $\mathcal{C}^p$-smooth.
        \item [{\rm (ii)}] The Jacobian $\nabla G(\overline{\bm{x}})$ has full row rank.
        \item [{\rm (iii)}] The following equality holds:
       $
           \mathcal{M} \cap \hat{\mathcal{B}}_{\overline{\bm{x}}}(\epsilon) = \{\bm{x} \in \hat{\mathcal{B}}_{\overline{\bm{x}}}(\epsilon) : G(\bm{x}) = \bm{0}\}.
       $
    \end{enumerate}
    We refer to $G(\bm{x}) = 0$ as the locally defined equation of $\mathcal{M}$.
\end{definition}
%This definition characterizes a $\mathcal{C}^p$-smooth manifold $\mathcal{M}$ of dimension $r$ around a point $\overline{\bm{x}}$ in terms of the existence of a $\mathcal{C}^p$-smooth mapping $\bm{G}$ from $\mathbb{R}^n$ to $\mathbb{R}^{n-r}$, such that the manifold is locally described as the set of points satisfying $\bm{G}(\bm{x}) = \bm{0}$, and the derivative $\nabla \bm{G}(\overline{\bm{x}})$ has full {row} rank. 
We define the tangent and normal spaces to $\mathcal{M}$ at $\bm{x}$ as $T_{\mathcal{M}}(\bm{x}) := \mbox{Null}(\nabla G(\bm{x}))$ and $N_{\mathcal{M}}(\bm{x}) := (T_{\mathcal{M}}(\bm{x}))^{\perp}$. Next, we introduce the concept of a partly smooth function and an active manifold \cite[Definition 2.7]{lewis2002active}.
\begin{definition}[Partly smooth function and active manifold]\label{def: Active manifold}
    Let $f : \mathbb{R}^n \to \mathbb{R} \cup \{\infty\}$ be a closed function, and $\mathcal{M} \subseteq \mathbb{R}^n$ be a set containing a point $\overline{\bm{x}}$. The function $f$ is said to be partly smooth at $\overline{\bm{x}}$ relative to $\mathcal{M}$ if:
        \begin{itemize}
        \item [{\rm (i)}] (Smoothness) The set $\mathcal{M} \cap \hat{\mathcal{B}}_{\overline{\bm{x}}}(\epsilon)$ is a $\mathcal{C}^p$-smooth manifold, and the restriction of $f$ to $\mathcal{M} \cap \hat{\mathcal{B}}_{\overline{\bm{x}}}(\epsilon)$ is $\mathcal{C}^p$-smooth.
        \item [{\rm (ii)}] (Regularity) For all $\bm{x}$ near $\overline{\bm{x}}$ in $\mathcal{M}$, $f$ is regular and has a subgradient.
        \item [{\rm (iii)}] (Normal sharpness) $dh(\bm{x})(-\bm{\nu}) > -dh(\bm{x})(\bm{\nu})$ for all nonzero $\bm{\nu}$ in $N_{\mathcal{M}}(\overline{\bm{x}})$.
        \item [{\rm (iv)}] (Subgradient continuity) The subdifferential mapping $\partial_c f$ is continuous at $\overline{\bm{x}}$ relative to $\mathcal{M}$.
    \end{itemize}
The manifold $\mathcal{M}$ is called a $\mathcal{C}^p$-active manifold, and $f$ admits $\mathcal{M}$ at $\overline{\bm{x}}$.
\end{definition}

Partly smooth functions encompass a wide variety of functions, such as smooth functions, polyhedral functions, the indicator and distance functions of a smooth manifold $\mathcal{M}$, and the sum of a smooth function and a partly smooth function. A more detailed introduction is provided in \cite{lewis2002active}. Numerous examples, particularly in signal processing, machine learning, and statistics, are discussed in \cite{vaiter2017model}. We summarize several examples in Table \ref{f admit active manifold} \cite{liang2017activity}. Figure \ref{fig:Common functions and their active manifolds} illustrates the landscapes of several widely used nonsmooth functions. As shown, the $\ell_1$, $\ell_1-\ell_2$, and $\ell_p$ functions admit a $\mathcal{C}^{\infty}$-active manifold at $(1,0)$ with $\mathcal{M} = \{(x,y) : y = 0\}$. The $\ell_{\infty}$ function and TV regularization admit a $\mathcal{C}^{\infty}$-active manifold at $(1,1)$ with $\mathcal{M} = \{(x,y) : x = y\}$. When restricted to $\mathcal{M}$ at the considered points, these functions vary smoothly along $\mathcal{M}$ locally, while grow sharply when moving normal to $\mathcal{M}$. We also introduce a class of nonconvex sparsity regularization functions that admit active manifolds in Table \ref{tab1}.

\begin{table}[htbp]
    \renewcommand\arraystretch{1.1}
    %\scalebox{1.5}{
      \begin{center}
    \begin{tabular}{@{}llllll@{}}
    %\begin{tabular}{@{}lll@{}}
    \hline
      Function  & Active manifold &  Function  & Active manifold  \\
    \hline
      $\|\cdot\|_1$   & $%\mathcal{M} = 
    \{\bm{z} \in \mathbb{R}^n : I_{\bm{z}} \subseteq I_{\overline{\bm{x}}}\}$ &
     $\|\cdot\|_\infty$   & $%\mathcal{M} = 
    \{\bm{z} \in \mathbb{R}^n : \bm{z}_{J_{\overline{\bm{x}}}} \in \mathbb{R}\mbox{sign}(\overline{\bm{x}}_{J_{\overline{\bm{x}}}})\}$    
 \\ 

        $\|\cdot\|_1-\|\cdot\|_2$ & %$\mathcal{M} = 
    $\{\bm{z} \in \mathbb{R}^n : I_{\bm{z}} \subseteq I_{\overline{\bm{x}}}\}$ & 
         $\|D_{\textbf{DIF}}\cdot\|_1$   & $%\mathcal{M} = 
    \{\bm{z} \in \mathbb{R}^n : I_{D_{\textbf{DIF}}\bm{z}}  \subseteq I_{D_{\textbf{DIF}}\overline{\bm{x}}}\} $   \\

    $\|\cdot\|_p^p$   & {$\{\bm{z} \in \mathbb{R}^n : I_{\bm{z}} \subseteq I_{\overline{\bm{x}}}\}$} &
    
    $\chi_{[\bm{l},\bm{u}]}(\cdot)$ & {$\{\bm{z} \in \mathbb{R}^n : \forall \overline{x}_i = u_i \:\mbox{or}\: l_i, z_i = \overline{x}_i \}$} \ \\
    \hline
    \end{tabular}
  \end{center}
  % \vspace{-2pt}
    \caption{Common functions and their active manifolds $\mathcal{M}$ at a given point $\overline{\bm{x}}$. We define $I_{\overline{\bm{x}}} = \{i : \overline{x}_i \neq 0\}$ and $J_{\overline{\bm{x}}} = \{i : |\overline{x}_i| = \|\overline{\bm{x}}\|_{\infty}\}$. We use $D_{\textbf{DIF}}\bm{x} = [0, x_2 - x_1, \dots, x_n - x_{n-1}]$ to represent the finite difference of $\bm{x}$ \cite{rudin1992nonlinear}. The notation $\mathbb{R}\text{sign}(\bm{x}_{I_{\bm{x}}})$ refers to the span of sign$(\bm{x}_{I_{\bm{x}}})$, and {$-\infty \leq \bm{l} \leq \bm{u} \leq +\infty$}. }
    \label{f admit active manifold}
    % \vspace{-10pt}
\end{table}

\begin{figure}[htbp]
% \vspace{-8pt}
\centering
\subfloat[%$\ell_1$ function: 
$ |x| + |y|$]{\label{fig:a}\includegraphics[width=0.3\textwidth]{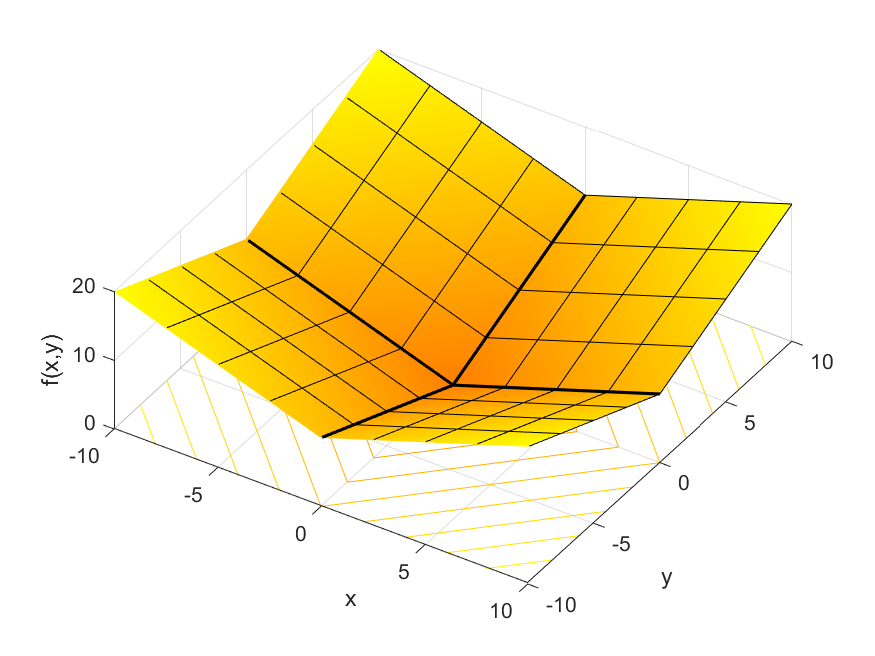}}\quad\quad
\subfloat[%$\ell_1-\ell_2$ function: 
$ |x| + |y|-\sqrt{x^2+y^2}$]{\label{fig:b}\includegraphics[width=0.3\textwidth]{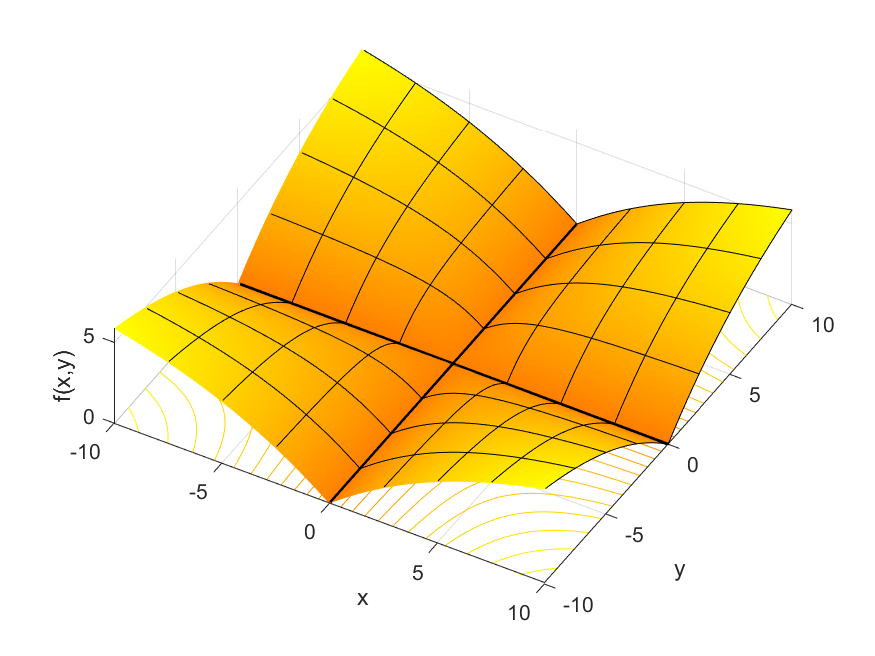}}\\
% \vspace{-10pt}
\subfloat[%$\ell_p$ function: 
$ |x|^{0.75} + |y|^{0.75}$]{\label{fig:e}\includegraphics[width=0.3\textwidth]{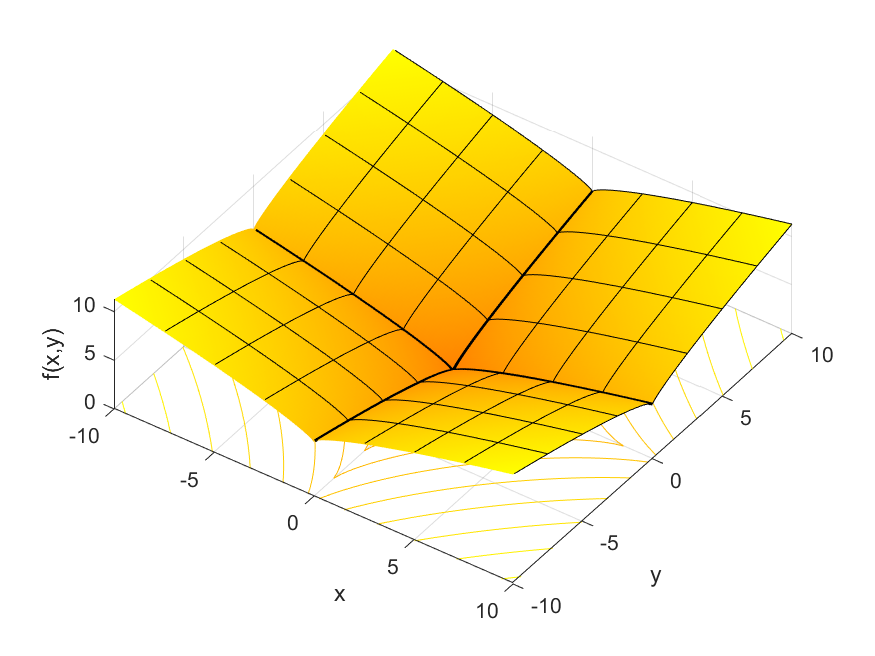}}\quad
\subfloat[%$\ell_{\infty}$ function: 
$ \max(|x|,|y|)$ ]{\label{fig:c}\includegraphics[width=0.3\textwidth]{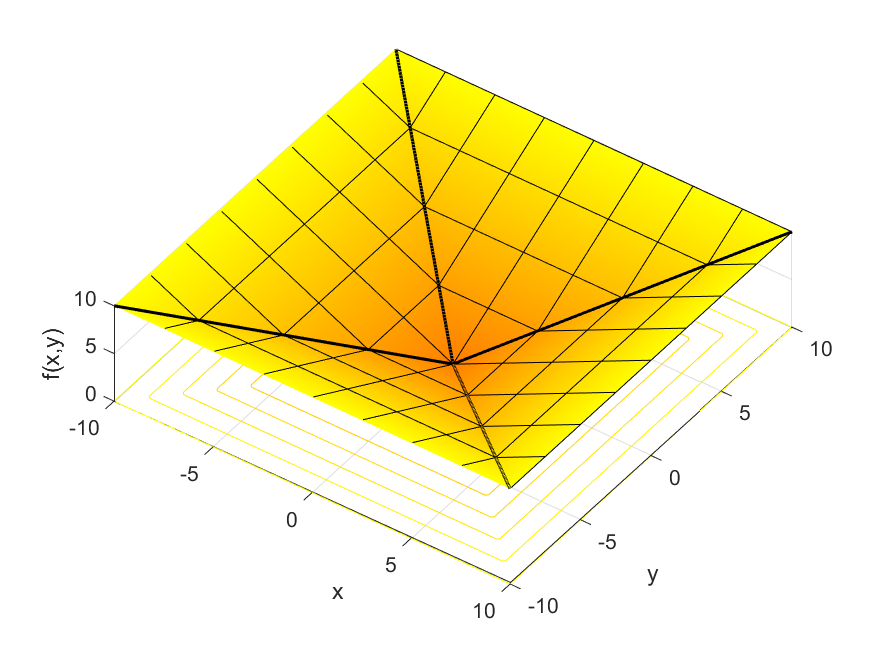}} \quad 
%% \vspace{-10pt}
\subfloat[%\red{TV function}: 
$ |y-x|$ ]{\label{fig:d}\includegraphics[width=0.3\textwidth]{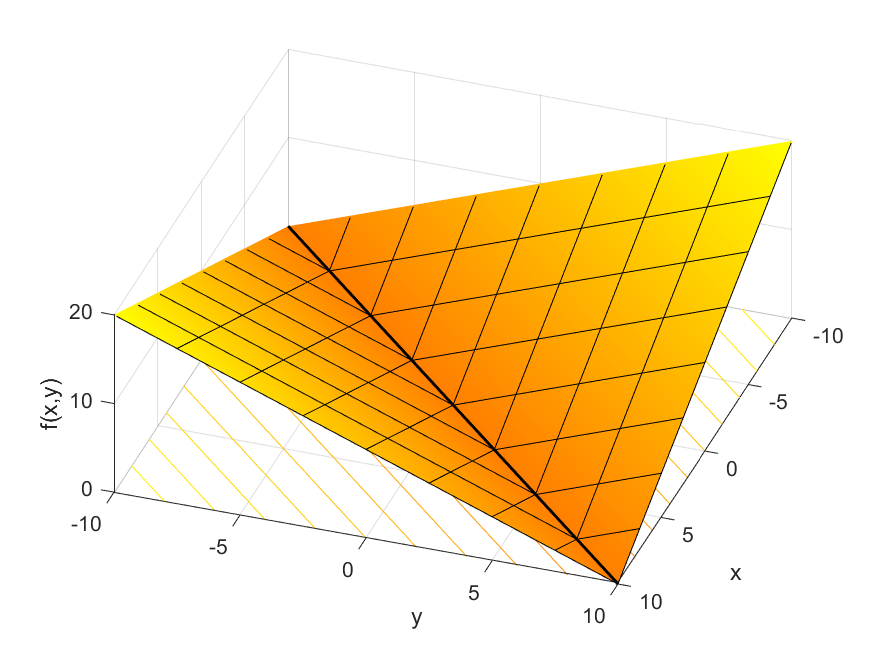}}
% \vspace{-15pt}
\caption{Landscapes of commonly used 2-dimensional nonsmooth functions $f(x,y)$.} 
\label{fig:Common functions and their active manifolds}
% \vspace{-5pt}
\end{figure}
% \vspace{-6pt}
For a partly smooth function $f$ and its active manifold $\mathcal{M}$, a key research question is whether an algorithm can identify the active manifold $\mathcal{M}$ within a finite number of steps. To ensure this, the non-degeneracy condition $\bm{0} \in \mathop{{\rm rint}}\partial_c f(\bm{x})$ is often imposed \cite{liang2017activity,mai2020anderson}. As argued in \cite{hare2007identifying}, this condition is nearly essential for guaranteeing the identification property. With the non-degeneracy condition, it has been shown that many algorithms can successfully identify the active manifold in a finite number of steps \cite{atenas2024weakly,davis2022proximal,hare2007identifying}. We now provide a formal definition:

% \vspace{-4pt}
\begin{definition}[Active manifold identification]\label{Active manifold identification}
    Consider an algorithm solving problem \eqref{Nonsmooth optimization problem}, which generates a sequence $\bm{x}^k \rightarrow \overline{\bm{x}}$, and whose iterations follow the fixed-point scheme $\bm{x}^{k+1} = H(\bm{x}^k)$. Suppose that $F$ admits an active manifold $\mathcal{M}$ at $\overline{\bm{x}}$. Then, we say the algorithm identifies the active manifold $\mathcal{M}$ if the inclusion $H(\bm{x}) \in \mathcal{M}$ holds for all $\bm{x}$ near $\overline{\bm{x}}$.
\end{definition}
%By Definition \ref{Active manifold identification}, it holds that $\bm{x}^k \in \mathcal{M}$ for all sufficiently large $k$. 
% \vspace{-4pt}
This property is crucial in the subsequent analysis, as it ensures the local smoothness of the  iteration mappings defined for the algorithms under consideration. To proceed, we examine the family of optimization problems:
\begin{equation}\label{ iteration mappings}
    \setlength{\abovedisplayskip}{8pt}   % 上方间距
    \setlength{\belowdisplayskip}{5pt}   % 下方间距
    \zeta(\bm{x}) = \mathop{\arg\min}_{\bm{y} \in \{\bm{y} \mid G(\bm{x}, \bm{y}) = \bm{0}, J(\bm{x}, \bm{y}) \leq \bm{0}\}} v(\bm{x}, \bm{y}), \quad \text{and} \quad V(\bm{x}) = v(\bm{x}, \zeta(\bm{x})),
\end{equation}
where $v: \mathbb{R}^n \times \mathbb{R}^d \rightarrow \mathbb{R} \cup \{\infty\}$, $G: \mathbb{R}^n \times \mathbb{R}^d \rightarrow \mathbb{R}^r$, and $J: \mathbb{R}^n \times \mathbb{R}^d \rightarrow \mathbb{R}^c$ are $\mathcal{C}^p$-smooth ($p \geq 2$). Following sensitivity analysis theory \cite{fiacco1976sensitivity}, particularly \cite[Theorem 2.1, Corollary 4.1]{fiacco1976sensitivity}, we establish the local smoothness of $\zeta(\bm{x})$ and $V(\bm{x})$.

\begin{lemma}[Sensitivity analysis]
\label{Local smoothness of subproblem mappings}
    Consider the problem in \eqref{ iteration mappings} and fix a point $\overline{\bm{x}} \in \mathbb{R}^n$. Suppose that $\zeta(\overline{\bm{x}})$ is a strong local minimizer and a unique global minimizer, and that $\zeta(\bm{x})$ varies continuously near $\overline{\bm{x}}$. If the Jacobian $\nabla_{\bm{y}} G(\overline{\bm{x}}, \zeta(\overline{\bm{x}}))$ has full row rank and the inequality constraint is inactive at $(\overline{\bm{x}}, \zeta(\overline{\bm{x}}))$, i.e., $J_i(\overline{\bm{x}}, \zeta(\overline{\bm{x}})) < 0$ for all $i = 1, \ldots, c$, then $\zeta(\bm{x})$ is $\mathcal{C}^{p-1}$-smooth and $V(\bm{x})$ is $\mathcal{C}^{p}$-smooth for all $\bm{x}$ near $\overline{\bm{x}}$.
\end{lemma}

When considering Anderson acceleration for an algorithm solving problem \eqref{Nonsmooth optimization problem}, the key idea is to reformulate the iteration as a fixed-point iteration $\bm{x}^{k+1} = H(\bm{x}^k)$ and then apply Algorithm \ref{AA fixed-point}. By utilizing active manifold theory, we can offer a simplified convergence analysis for various algorithms. The active manifold identification property and Lemma \ref{Local smoothness of subproblem mappings} are pivotal. With these, we establish the local R-linear convergence rate for Anderson-accelerated algorithms. The central idea is to ensure the local smoothness of $H$, following the reasoning in \cite{davis2022proximal}. Theorem \ref{Convergence for AA prox-type algorithm} forms the foundation for the theoretical analysis in the subsequent sections.

\begin{theorem}\label{Convergence for AA prox-type algorithm}
  Consider the problem \eqref{Nonsmooth optimization problem}, and let $\bm{x}^*$ be a Clarke critical point of $F$. Suppose that $F$ admits a $\mathcal{C}^{3}$-active manifold $\mathcal{M}$ at $\bm{x}^*$. Assume an algorithm generates a sequence $\bm{x}^k \rightarrow \bm{x}^*$, with iterations formulated as
  \begin{equation}\label{FPI}
      \setlength{\abovedisplayskip}{8pt}   % 上方间距
      \setlength{\belowdisplayskip}{5pt}   % 下方间距
     \bm{x}^{k+1} = H(\bm{x}^k), \quad \text{where} \quad H(\bm{x}) := \mathop{\arg\min}_{\bm{y} \in \mathbb{R}^n} s(\bm{x}, \bm{y}).
  \end{equation}
  Here $s: \mathbb{R}^n \times \mathbb{R}^n \to \mathbb{R} \cup \{\infty\}$ inherits the $\mathcal{C}^3$-smoothness of $F$ on $\mathcal{M}$ near $\bm{x}^*$. Assume that $H$ is a unique strong global minimizer and the algorithm identifies $\mathcal{M}$, the following statements hold:
  \begin{itemize}
  \item[(i)] $H$ is $\mathcal{C}^{2}$-smooth, and both $H$ and $\nabla H$ are Lipschitz continuous near $\bm{x}^*$.
  \item[(ii)] Under Assumption \ref{Assumption AA}, if the initial point $\bm{x}^0$ is sufficiently close to $\bm{x}^*$, applying Anderson acceleration to \eqref{FPI} via Algorithm \ref{AA fixed-point} yields iterates that converge to $\bm{x}^*$ R-linearly with a rate $\hat{\gamma} \in (\gamma, 1)$, i.e., \eqref{cov-rate} holds.
  \end{itemize}
  \end{theorem}
  
  \begin{proof}
  (i) The active manifold identification property ensures that $H(\bm{x}) \in \mathcal{M}$ for all $\bm{x}$ near $\bm{x}^*$. Thus, we can approach the analysis via sensitivity analysis. Let $\hat{s}(\bm{x}, \cdot)$ be a $\mathcal{C}^3$-smooth function that coincides with ${s}(\bm{x}, \cdot)$ on $\mathcal{M}$ near $\bm{x}^*$. For any $\bm{x}$ close to $\bm{x}^*$, we have 
  \begin{equation}\label{Subp_active}
      \setlength{\abovedisplayskip}{8pt}   % 上方间距
      \setlength{\belowdisplayskip}{5pt}   % 下方间距
     H(\bm{x}) = \mathop{\arg\min}_{\bm{y} \in \mathcal{M} \cap \mathcal{B}_{\bm{x}^*}(\hat{\rho})} \hat{s}(\bm{x}, \bm{y}),
  \end{equation}
  where $0<\hat{\rho}\le\rho$. Then we are ready to apply Lemma \ref{Local smoothness of subproblem mappings} to the problem \eqref{Subp_active} with 
  $$
      \setlength{\abovedisplayskip}{8pt}   % 上方间距
      \setlength{\belowdisplayskip}{5pt}   % 下方间距
      \zeta(\bm{x}) := H(\bm{x})\text{ and }   v(\bm{x} , \bm{y}):= \hat{s}(\bm{x}, \bm{y}).
  $$
  Additionally, by Definition \ref{def:2.2} (iii), there exists a mapping $G^{\mathcal{M}}: \mathbb{R}^n \to \mathbb{R}^{n-r}$ such that $\mathcal{M} \cap \mathcal{B}_{\bm{x}^*}(\hat{\rho}) = \{\bm{x} \in \mathcal{B}_{\bm{x}^*}(\hat{\rho}): G^{\mathcal{M}}(\bm{x}) = \bm{0}\}$. Hence, 
  $$
      \setlength{\abovedisplayskip}{8pt}   % 上方间距
      \setlength{\belowdisplayskip}{5pt}   % 下方间距
      \mathcal{M} \cap \mathcal{B}_{\bm{x}^*}(\hat{\rho}) = \{\bm{x} \in \mathbb{R}^n: G(\bm{x}, \bm{y}) = \bm{0}, J(\bm{x}, \bm{y}) \le \bm{0}  \},
  $$
  where $G(\bm{x}, \bm{y}) := G^{\mathcal{M}}(\bm{y})$ \text{ and }$ J(\bm{x}, \bm{y}) := \|\bm{y} - \bm{x}^*\|^2 - \hat{\rho}^2.$ 
  % $$
  %     \setlength{\abovedisplayskip}{8pt}   % 上方间距
  %     \setlength{\belowdisplayskip}{5pt}   % 下方间距
  %    \zeta(\bm{x}) := H(\bm{x}) = \mathop{\arg\min}_{\bm{y} \in \{\bm{y} \mid G(\bm{x}, \bm{y}) = \bm{0}, J(\bm{x}, \bm{y}) \le \bm{0}\}} \hat{s}(\bm{x}, \bm{y}),
  % $$
  Then we conclude that $H$ is $\mathcal{C}^2$-smooth near $\bm{x}^*$, which also implies that $H$ and $\nabla H$ are Lipschitz continuous.
  
  (ii) Building on the result from (i) and Assumption \ref{Assumption AA}, there exists a neighborhood $\mathcal{B}_{\bm{x}^*}(\rho)$ of $\bm{x}^*$ for some sufficiently small $\rho > 0$ such that $\|\nabla H(\bm{x})\| \le \gamma$ for all $\bm{x} \in \mathcal{B}_{\bm{x}^*}(\rho)$. By Theorem \ref{Convergence AA}, we establish the R-linear convergence of the Anderson-accelerated sequence, thus completing the proof.
  \end{proof}
  
  \vspace{-5pt}
  \section{Anderson-accelerated proximal algorithms} 
  Our task in this section is to establish the theoretical properties of Anderson acceleration applied to three proximal algorithms for solving the non-smooth optimization problem:
  \begin{equation} \label{Problem in proximal}
      \setlength{\abovedisplayskip}{8pt}   % 上方间距
      \setlength{\belowdisplayskip}{5pt}   % 下方间距
      \min\limits _{\bm{x}\in \mathbb{R}^n} F(\bm{x}) := (f \circ h)(\bm{x}) + g(\bm{x}),
  \end{equation}
  where $ h: \mathbb{R}^n \rightarrow \mathbb{R}^d $ is $\mathcal{C}^2$-smooth, $ f: \mathbb{R}^d \rightarrow \mathbb{R} $ is convex, and $ g: \mathbb{R}^n \rightarrow \mathbb{R} \cup \{\infty\} $ is closed and $\sigma$-weakly convex. Similar to \cite{lewis2016proximal}, we assume that there exists $ \eta > 0 $ such that:
  \begin{equation}\label{assumption for proximal linear}
      \setlength{\abovedisplayskip}{8pt}   % 上方间距
      \setlength{\belowdisplayskip}{5pt}   % 下方间距
      |f(h(\bm{y})) - f(h(\bm{x}) + \nabla h(\bm{x})(\bm{y} - \bm{x}))| \leq \frac{\eta}{2} \|\bm{y} - \bm{x}\|^2, \quad \forall \bm{x}, \bm{y} \in \mathbb{R}^n.
  \end{equation}
  % This leads to the conclusion that $ F $ is $(\eta + \sigma)$-weakly convex.
  
  Problem \eqref{Problem in proximal} and its variants encompass a wide array of non-smooth optimization issues. In Table \ref{Proximal algorithm table}, we present three types of non-smooth optimization problems along with commonly used algorithms, including the proximal point algorithm \cite{mordukhovich2006variational}, proximal gradient (Forward–Backward splitting) algorithm \cite{beck2009fast,liang2017activity}, and proximal linear algorithm \cite{lewis2016proximal}. Each of these algorithms involves subproblems that can be represented as fixed-point iterations in the form \eqref{FPI}.

  \begin{table}[htbp]
    % \fontsize{5pt}{20pt}\selectfont
        \renewcommand\arraystretch{1}
        %\scalebox{1.5}{\scalebox{0.8}{ % 调整表格整体大小
          \begin{center}
        \begin{tabular}{@{}lll@{}}
        %\begin{tabular}{@{}lll@{}}
        \hline
        Alg. & $F(\bm{x})$  & $s(\bm{x},\bm{y})$  \\
        \hline
        PPA    & $g(\bm{x})$   & $g(\bm{y}) + \frac{1}{2\beta}\|\bm{y} - \bm{x}\|^2$  \\
    
        PGA    & $h(\bm{x}) + g(\bm{x})$   & $h(\bm{x}) + \nabla h(\bm{x})(\bm{y} - \bm{x})) + g(\bm{y}) + \frac{1}{2\beta}\|\bm{y} - \bm{x}\|^2$ \\
            
        PLA     & $(f\circ h)(\bm{x}) + g(\bm{x})$   & $f(h(\bm{x}) + \nabla h(\bm{x})(\bm{y} - \bm{x})) + g(\bm{y}) + \frac{1}{2\beta}\|\bm{y} - \bm{x}\|^2$ \\
        \hline
        \end{tabular}
      \end{center}
      \vspace{-3pt}
        \caption{ Three proximal-type algorithms  for solving  \eqref{Problem in proximal}. {PPA: Proximal point algorithm; PGA: Proximal gradient algorithm; PLA: Proximal linear algorithm}}\label{Proximal algorithm table}
        \vspace{-10pt}
    \end{table}
    \vspace{-0pt}
    %Based on Theorem \ref{Convergence for AA prox-type algorithm}, to establish Anderson acceleration for proximal type algorithms it suffices to 
    % In subsequent subsections we will present the theoretical properties of Anderson accelerations for specific proximal type algorithms applied to \eqref{Problem in proximal} in different scenarios. 
    As illustrated in Table \ref{Proximal algorithm table}, the proximal linear algorithm generalizes both the proximal point algorithm (with $ h \equiv 0 $, $ f \equiv 0 $) and the proximal gradient algorithm (with $ d = 1 $, $ f = \mathbb{I}_d $). Moreover, in the context of the proximal gradient algorithm, condition \eqref{assumption for proximal linear} can be inferred from the Lipschitz continuity of $ \nabla h $. We will focus on the proximal linear algorithm and provide a convergence analysis for the Anderson-accelerated proximal linear algorithm. The results in Theorem \ref{Theorem AA proxlinear} can also be applied to Anderson-accelerated proximal point and proximal gradient algorithms, 
    
    We are now ready to demonstrate the locally R-linear convergence rate of the Anderson-accelerated proximal linear algorithm. Theorem \ref{Convergence for AA prox-type algorithm} plays a critical role in the subsequent analysis, enabling us to show the smoothness of the iteration mapping of the proximal linear algorithm around \( \bm{x}^* \) and establish the local R-linear convergence rate for the Anderson-accelerated proximal linear algorithm. Before presenting the main result of this section, we introduce the concept of the composite active manifold for the composite function \( F \) in problem \eqref{Problem in proximal}, based on the active manifolds of \( f \) and \( g \) and the transversality condition—a fundamental notion in differential geometry that guarantees the intersection of two manifolds remains a manifold \cite[Theorem 6.30]{lee2012smooth} \cite[Definition 5.1]{davis2022proximal}. 
    % The revision intuitively requires the sharpness condition to be satisfied for the original function with a linear tilt
    % First, we need a more generalized active manifold. Consider a set $\mathcal{G} \subset \mathbb{R}^n$, a closed weakly convex function $g:\R^n\to \R\cap\{\infty\}$, a point $\bm{x} \in \mathcal{G}$, and  $\bm{v} \in \partial g(\bm{x})$. We call $\mathcal{G}$ a $\mathcal{C}^p$-active manifold of $g$ at $\bm{x}$ for $\bm{v}$ if $\mathcal{G}$ is a $\mathcal{C}^p$-active manifold of the function $g - \langle \bm{v}, \cdot \rangle$ at $\bm{x}$. With the revised active manifold of $g$ and $f$, we introduce the concept of a composite active manifold \cite[Definition 5.1]{davis2022proximal} for the problem \eqref{Problem in proximal}. 
    \begin{definition}[Composite active manifold]\label{Composite active manifold}
    Consider the composite optimization problem \eqref{Problem in proximal}. Let $ \bm{x}^* $ be a critical point of $ F $. For any $ \bm{w}^* \in \partial f(h(\bm{x}^*)) $ and $ \bm{v}^* \in \partial g(\bm{x}^*) $ satisfying
    $\bm{0} \in \nabla h(\bm{x}^*)^* \bm{w}^* + \bm{v}^*,$ suppose the following conditions hold:
    \begin{enumerate}
        \item [{\rm (i)}] There exist $ \mathcal{C}^p $-smooth manifolds $ \mathcal{G} \subset \mathbb{R}^n $ and $ \mathcal{F} \subset \mathbb{R}^d $ containing $ \bm{x}^* $ and $ h(\bm{x}^*) $, respectively, satisfying the transversality condition:
        \begin{equation} \label{Transversality Condition}
            \setlength{\abovedisplayskip}{8pt}   % 上方间距
            \setlength{\belowdisplayskip}{5pt}   % 下方间距
            \nabla h(\bm{x}^*)[T_{\mathcal{G}}(\bm{x}^*)] + T_{\mathcal{F}}(h(\bm{x}^*)) = \mathbb{R}^d.
        \end{equation}
        \item [{\rm (ii)}] Non-degeneracy conditions $ \bm{v}^* \in \mbox{rint}(\partial g(\bm{x}^*)) $ and $ \bm{w}^* \in \mbox{rint}(\partial f(h(\bm{x}^*))) $ hold.
        \item [{\rm (iii)}] $ \mathcal{G} $ is an active manifold of $ g $ at $ \bm{x}^* $, and $ \mathcal{F} $ is an active manifold of $ f $ at $ h(\bm{x}^*) $.
    \end{enumerate}
    We say that $ F $ admits a $ \mathcal{C}^p $-composite active manifold $ \mathcal{M} := \mathcal{G} \cap h^{-1}(\mathcal{F}) $ at $ \bm{x}^* $.
    \end{definition}
    
    According to Definition \ref{Composite active manifold}, to ensure that $ F $ admits an active manifold at $ \bm{x}^* $, the nonsmooth functions $ g $ and $ f $ must admit active manifolds $ \mathcal{G} $ and $ \mathcal{F} $ at $ \bm{x}^* $ and $ h(\bm{x}^*) $, respectively. The non-degeneracy conditions for $ f $ and $ g $, along with the classical transversality condition \eqref{Transversality Condition} ensure $ \bm{0} \in \mbox{rint}(\partial F(\bm{x}^*)) $ \cite{davis2022proximal}.
    
    \begin{theorem}\label{Theorem AA proxlinear}
    Consider the composite optimization problem \eqref{Problem in proximal}. Let $ \bm{x}^* $ be a critical point of $ F $ and let
    \begin{equation}\label{FPI prox-linear}
        \setlength{\abovedisplayskip}{8pt}   % 上方间距
        \setlength{\belowdisplayskip}{5pt}   % 下方间距
        H(\bm{x}) = \mathop{\arg\min}_{\bm{y}\in\mathbb{R}^n}  f(h(\bm{x}) + \nabla h(\bm{x})(\bm{y} - \bm{x})) + g(\bm{y}) + \frac{1}{2\beta}\|\bm{y} - \bm{x}\|^2,
    \end{equation}
    where 
    $ \beta \in (0, \min\{\sigma^{-1},\eta^{-1}\}) $
    % $ \beta \in (0, \min\{\sigma^{-1},L_f^{-1}\}) $
    % $ \beta \in (0, (\sigma + \eta)^{-1}) $
    . Suppose that $ F $ admits a $ \mathcal{C}^3 $-composite active manifold $ \mathcal{M} $ as defined in Definition \ref{Composite active manifold} and that $ h $ is $ \mathcal{C}^4 $-smooth near $ \bm{x}^* $.
    \footnote{To apply the sensitivity analysis theory for establishing the \( \mathcal{C}^2 \)-smoothness of \( H \), higher-order smoothness of \( h \) is required to ensure that the functions in \eqref{PLA_Subp_active} are \( \mathcal{C}^3 \)-smooth.}
    Then:
    \begin{itemize}
    \item[(i)] $ H $ is $ \mathcal{C}^2 $-smooth, and both $ H $ and $ \nabla H $ are Lipschitz continuous around $ \bm{x}^* $.
    \item[(ii)] Under Assumption \ref{Assumption AA}, if the initial point $ \bm{x}^0 $ is sufficiently close to $ \bm{x}^* $, then applying Anderson acceleration to \eqref{FPI prox-linear} according to Algorithm \ref{AA fixed-point} results in iterates that converge to $ \bm{x}^* $ R-linearly with $ \hat{\gamma} \in (\gamma, 1) $; that is, \eqref{cov-rate} holds.
    \end{itemize}
    \end{theorem}
    
    \begin{proof}
    The proximal linear algorithm has been shown to possess the active manifold identification property \cite[Theorem 4.11]{lewis2016proximal}. Thus, the inclusions \cite[Theorem 5.3]{davis2022proximal} $
    H(\bm{x}) \in \mathcal{G}$ \text{and} $ h(\bm{x}) + \nabla h(\bm{x})(H(\bm{x}) - \bm{x}) \in \mathcal{F}$
    hold for all $ \bm{x} $ near $ \bm{x}^* $. We can establish the smoothness of $ H $ through sensitivity analysis. Consider any $ \mathcal{C}^3 $-smooth functions $ \hat{f}:\mathbb{R}^d\rightarrow \mathbb{R} $ and $ \hat{g}:\mathbb{R}^n\rightarrow \mathbb{R} $ that agree with $ f $ and $ g $ near $ h(\bm{x}^*) $ and $ \bm{x}^* $ on the smooth manifolds $ \mathcal{F} $ and $ \mathcal{G} $. Then $ H(\bm{x}) $ solves the smooth problem near $ \bm{x}^* $:
    \begin{equation}\label{PLA_Subp_active}
        \setlength{\abovedisplayskip}{8pt}   % 上方间距
        \setlength{\belowdisplayskip}{5pt}   % 下方间距
        \begin{aligned}
        H(\bm{x}) = &\arg\min_{\bm{y}} \hat{f}(h(\bm{x}) + \nabla h(\bm{x})(\bm{y} - \bm{x})) + \hat{g}(\bm{y}) + \frac{1}{2\beta}\|\bm{y} - \bm{x}\|^2 \\
        &\quad\quad\,\, \text{s.t. } \bm{y} \in \mathcal{G} \cap \mathcal{B}_{\bm{x}^*}(\hat{\rho}), \quad h(\bm{x}) + \nabla h(\bm{x})(\bm{y} - \bm{x}) \in \mathcal{F} \cap \mathcal{B}_{h(\bm{x}^*)}(\hat{\rho}),
        \end{aligned}
    \end{equation}
    where $ \mathcal{B}_{\bm{x}^*}(\hat{\rho}) $ and $ \mathcal{B}_{h(\bm{x}^*)}(\hat{\rho}) $ are neighborhoods of $ \bm{x}^* $ and $ h(\bm{x}^*) $ with radius $0< \hat{\rho}\le\rho $. By Definition \ref{def:2.2}, there exist mappings $ G^{\mathcal{G}} $ and $ G^{\mathcal{F}} $ that identify the locally defined equations of $ \mathcal{G} $ and $ \mathcal{F} $. Thus, problem \eqref{PLA_Subp_active} can be reformulated as:
    $$
        \setlength{\abovedisplayskip}{2pt}
        \setlength{\belowdisplayskip}{2pt}
        \begin{aligned}
            \zeta(\bm{x}) &:= H(\bm{x}), \quad v(\bm{x}, \bm{y}) := \hat{f}(h(\bm{x}) + \nabla h(\bm{x})(\bm{y} - \bm{x})) + \hat{g}(\bm{y}) + \frac{1}{2\beta}\|\bm{y} - \bm{x}\|^2, 
        \end{aligned}
    $$
    where the constraints 
    $$
        \setlength{\abovedisplayskip}{2pt}
        \setlength{\belowdisplayskip}{2pt}
        \bm{y} \in \mathcal{G} \cap \mathcal{B}_{\bm{x}^*}(\hat{\rho}) \quad \text{and} \quad h(\bm{x}) + \nabla h(\bm{x})(\bm{y} - \bm{x}) \in \mathcal{F} \cap \mathcal{B}_{h(\bm{x}^*)}(\hat{\rho})
    $$ 
    are equivalent to the condition 
    $
        \{\bm{x} \in \mathbb{R}^n: G(\bm{x}, \bm{y}) = \bm{0}, J(\bm{x}, \bm{y}) \leq \bm{0}\}
    $
    with
    $$
        \setlength{\abovedisplayskip}{2pt}
        \setlength{\belowdisplayskip}{2pt}
        \begin{aligned}
            G(\bm{x}, \bm{y}) &:= \left[\begin{array}{c}
                G^{\mathcal{G}}(\bm{y}) \\
                G^{\mathcal{F}}(h(\bm{x}) + \nabla h(\bm{x})(\bm{y} - \bm{x}))
            \end{array}\right], \\
            J(\bm{x}, \bm{y}) &:= \left[\begin{array}{c}
                \|\bm{y} - \bm{x}^*\|^2 - \rho^2 \\
                \|h(\bm{x}) + \nabla h(\bm{x})(\bm{y} - \bm{x}) - h(\bm{x}^*)\|^2 - \rho^2
            \end{array}\right].
        \end{aligned}
    $$
    To apply Lemma \ref{Local smoothness of subproblem mappings}, we first verify its assumptions. The continuity of $ \zeta $ is ensured by Lemma \ref{Moreau envelope and the proximal point mapping} under the condition $ \beta \in (0, \min\{\sigma^{-1},\eta^{-1}\})$, which implies that $ \zeta(\bm{x}^*) $ is a strong global minimizer. The properties of the active manifolds $ \mathcal{G} $ and $ \mathcal{F} $ and the transversality condition \eqref{Transversality Condition} guarantee that the Jacobian $ \nabla_{\bm{y}}G $ has full row rank at $ (\bm{x}^*,\zeta(\bm{x}^*)) $. Furthermore, the inequality constraint $ J $ is naturally inactive at $ (\bm{x}^*, \zeta(\bm{x}^*)) $. By applying Lemma \ref{Local smoothness of subproblem mappings}, we conclude that $ H $ is $ \mathcal{C}^2 $-smooth around $ \bm{x}^* $. Similar to the analysis in Theorem \ref{Convergence for AA prox-type algorithm}, we complete the proof.
    \end{proof}

    \begin{remark}
      There has been research on saddle point problems that demonstrates the smoothness of the mapping $ H $ as defined in \eqref{FPI prox-linear}. More details can be found in \cite[Theorem 5.1]{davis2022proximal}. Moreover, it is worth noting that Theorem \ref{Theorem AA proxlinear} is also applicable to PPA and PGA with their corresponding simplified problems (as shown in Table \ref{Proximal algorithm table}). Furthermore, Definition \ref{Composite active manifold} regarding composite active manifold reduces to Definition \ref{def: Active manifold} in the context of PPA and PGA, with the non-degeneracy condition $ \bm{0} \in \text{rint}(\partial F(\bm{x}^*)) $.
      \end{remark} 
      \section{Anderson-accelerated proximal coordinate descent algorithm}
      Consider the problem:
      \begin{equation} \label{Problem coordinate descent algorithm}
          \setlength{\abovedisplayskip}{8pt}   % 上方间距
          \setlength{\belowdisplayskip}{5pt}   % 下方间距
        \min\limits _{\bm{x}\in \mathbb{R}^n} F(\bm{x}):= f(\bm{x}) +\displaystyle\sum^n_{i=1} g_i ( x_i),
      \end{equation}
      where $ f:\mathbb{R}^n\rightarrow \mathbb{R} $ is a $\mathcal{C}^2$-smooth convex function with $ L_f $-Lipschitz continuous gradients, and $ g_i:\mathbb{R}\rightarrow \mathbb{R}\cup \{\infty\}, i=1, \ldots,n $ are proper, closed and convex. A common algorithm for solving such problems is the proximal coordinate descent algorithm (PCD) \cite{fercoq2015accelerated,wright2015coordinate}. In this paper, we focus on PCD in a cyclic scheme, which selects coordinate indices cyclically from the set $ \{1,\ldots,n\} $. The  iteration can be expressed as follows: 
      \begin{equation}
          \setlength{\abovedisplayskip}{8pt}   % 上方间距
          \setlength{\belowdisplayskip}{5pt}   % 下方间距
          \begin{aligned}\label{CPCD}
        \begin{cases}
          \bm{x}^{k,1} &=  \bm{x}^{k} + ({\rm prox}_{\beta g_1} (x^k_1 - \beta\nabla_1 f(\bm{x}^k))-x^k_1)e_1, \\
          \vdots &\\
          \bm{x}^{k+1} &=  \bm{x}^{k,n-1} + ({\rm prox}_{\beta g_n} (x^{k,n-1}_n - \beta\nabla_n f(\bm{x}^{k,n-1}))-x^{k,n-1}_n)e_n,
        \end{cases}
      \end{aligned}
      \end{equation}
      Introduce $ \psi_j:\mathbb{R}^n\rightarrow \mathbb{R}^n $ with the $ i $th component defined by
      \begin{equation}
          \setlength{\abovedisplayskip}{8pt}   % 上方间距
          \setlength{\belowdisplayskip}{5pt}   % 下方间距
          \begin{aligned}
              \relax    [\psi_j(\bm{x})]_i :=
            \begin{cases}
             {\rm prox}_{\beta g_i} (x_i - \beta\nabla_i f(\bm{x}))  & i = j, \\
                  x_i  & i \neq j,
            \end{cases}
          \end{aligned}
      \end{equation}
      where $ i,j=1,\ldots,n $. 
      Given a fixed constant $ \beta \in(0, L_f^{-1}) $, \eqref{CPCD} can be interpreted as a fixed-point iteration:
      \begin{equation}
          \setlength{\abovedisplayskip}{8pt}   % 上方间距
          \setlength{\belowdisplayskip}{5pt}   % 下方间距
          \begin{aligned}\label{FPI coordinate}
          \bm{x}^{k+1} = H(\bm{x}^k) :=\psi_n \circ \cdots  \circ \psi_1 (\bm{x}^k)=        \left[\begin{array}{c}
                  [\psi_1(\bm{x}^k)]_1  \\
                  {[(\psi_2\circ\psi_1)(\bm{x}^k)]}_2 \\
                  \cdots \\
                  {[(\psi_n \circ \cdots  \circ \psi_1)(\bm{x}^k)]}_n
                  \end{array}\right].
      \end{aligned}
      \end{equation}
      
      We also demonstrate the smoothness of the mapping $ H $ around $ \bm{x}^* $ through the active manifold identification property and Lemma \ref{Local smoothness of subproblem mappings}, and establish the locally R-linear convergence rate of Anderson-accelerated PCD.
      
      \begin{theorem}
         Consider the optimization problem \eqref{Problem coordinate descent algorithm}. Let $ \bm{x}^* $ be a critical point of $ F $ and $ H = \psi_n \circ \cdots  \circ \psi_1 $ with $ \beta \in(0, L_f^{-1}) $. Suppose that $ g $ admits a $ C^3 $-active manifold $ \mathcal{M} $ at $ \bm{x}^* $, $ f $ is $\mathcal{C}^4$-smooth near $ \bm{x}^* $, and $ \bm{0} \in \mbox{rint}(\partial F(\bm{x^*})) $, then:
      \begin{itemize}
      \item[(i)]  $ H $ is $\mathcal{C}^{2}$-smooth, and both $ H $ and $ \nabla H $ are Lipschitz continuous around $ \bm{x}^* $.
      \item[(ii)] Under Assumption \ref{Assumption AA}, if the initial point $ \bm{x}^0 $ is sufficiently close to $ \bm{x}^* $, applying Anderson acceleration to \eqref{FPI coordinate} according to Algorithm \ref{AA fixed-point} leads to iterates that converge to $ \bm{x}^* $ R-linearly with $ \hat{\gamma}\in(\gamma,1) $, i.e., \eqref{cov-rate} holds.
      \end{itemize}
      \end{theorem}
      \begin{proof}
      Given that $ F $ admits the active manifold $ \mathcal{M} $, previous research has established the active manifold identification property of PCD \cite{klopfenstein2024local}. Specifically, , we have $ H(\bm{x})\in\mathcal{M} $ for all $\bm{x}$ near $\bm{x}^*$. However, we cannot directly establish the smoothness of $ H $ through Lemma \ref{Local smoothness of subproblem mappings}, as the mapping $ H $ does not correspond to a strongly convex function $ s(\bm{x},\cdot) $. Instead, we will infer the smoothness of $ H $ by analyzing the smoothness of $ \psi:=\psi_n\circ\cdots\circ\psi_1 $.
      
      Let $ \tilde{g} := \sum^n_{i=1} \tilde{g}_i $ with $ \tilde g_i $ being $ \mathcal{C}^3 $-smooth and coinciding with $ g_i $ for $ i=1,\ldots,n $ on $ \mathcal{M}$ near $\bm{x}^*$. Consequently, the component $ [\psi_1(\bm{x})]_1 $ uniquely solves the following optimization problem:
      $$
          \setlength{\abovedisplayskip}{8pt}   % 上方间距
          \setlength{\belowdisplayskip}{5pt}   % 下方间距
          \min_{y_1,\bm{y}\in\mathcal{M}\cap \mathcal{B}_{\bm{x}^*}(\hat{\rho})} s(\bm{x},y_1): = \nabla_1 f(\bm{x}) y_1 + \tilde{g}_1(y_1) + \frac{1}{2\beta} (y_1 - x_1)^2,
      $$
      where $\mathcal{B}_{\bm{x}^*}(\hat{\rho})$ is a neighborhood of $\bm{x}^*$ with radius $0<\hat{\rho}\le \rho$. Given any $ \bm{x}\in \mathbb{R}^n $, $ s(\bm{x}, y_1) $ is strongly convex with respect to $ y_1 $. Moreover, Lemma \ref{Moreau envelope and the proximal point mapping} implies that $ \psi $ is continuous. Thus, similar to the proof in Theorem \ref{Convergence for AA prox-type algorithm}, we can apply Lemma \ref{Local smoothness of subproblem mappings} to conclude that $ \psi_1 $ is $ \mathcal{C}^2 $-smooth around $ \bm{x}^* $. 
      
      Similarly, the component $ {[(\psi_2\circ\psi_1)(\bm{x})]}_2 $ uniquely solves the problem:
      $$
          \setlength{\abovedisplayskip}{8pt}   % 上方间距
          \setlength{\belowdisplayskip}{5pt}   % 下方间距
          \min_{y_2,\bm{y}\in\mathcal{M}\cap \mathcal{B}_{\bm{x}^*}(\hat{\rho})}  \nabla_2 f(\psi_1(\bm{x})) y_2 + \tilde{g}_2(y_2) + \frac{1}{2\beta} (y_2 - x_2)^2.
      $$
      By applying Lemma \ref{Local smoothness of subproblem mappings}, we derive that $ [(\psi_2 \circ \psi_1)(\cdot)]_2 $ is also $ \mathcal{C}^2 $-smooth. Repeating this process, we obtain that $ [(\psi_n\circ \dots\circ\psi_1)(\cdot)]_n $ is smooth, thus $ H $ is $ \mathcal{C}^2 $-smooth around $ \bm{x}^* $. Similar to the analysis in Theorem \ref{Convergence for AA prox-type algorithm}, statement (ii) can be proved.
      \end{proof}

      \section{Anderson-accelerated {DRS/ADMM}}
      % We now focus on the Anderson-accelerated Douglas-Rachford (DR) splitting algorithm. 
      We now consider the problem:
      \begin{equation} \label{Problem:DRS}
          \setlength{\abovedisplayskip}{8pt}   % 上方间距
          \setlength{\belowdisplayskip}{5pt}   % 下方间距
        \min\limits _{\bm{x}\in \mathbb{R}^n} F(\bm{x}):= f(\bm{x}) + g(\bm{x}),
      \end{equation}
      where $ f:\mathbb{R}^n\rightarrow \mathbb{R} $ is $\mathcal{C}^2$-smooth with $ L_f $-Lipschitz continuous gradients, and $ g:\mathbb{R}^n\rightarrow \mathbb{R}\cup \{\infty\} $ is closed and $\sigma$-weakly convex. The Douglas-Rachford splitting algorithm (DRS) \cite{fu2020anderson,themelis2020douglas} is a well-known method for solving such problems. Given $ \bm{z}^0\in\mathbb{R}^n $, $ \beta \in(0, \min\{\sigma^{-1},L_f^{-1}\}) $, and $ \delta \in (0,2) $, the iteration of DRS is performed as follows:
      $$    \setlength{\abovedisplayskip}{8pt}   % 上方间距
          \setlength{\belowdisplayskip}{5pt}   % 下方间距
          \begin{aligned}\label{Douglas-Rachford splitting}
              \begin{cases}
                  \bm{x}^{k+1} &= {\rm prox}_{\beta f}(\bm{z}^{k}), \\
                  \bm{y}^{k+1} &= {\rm prox}_{\beta g}(2\bm{x}^{k+1} - \bm{z}^{k}), \\
                  \bm{z}^{k+1} &= \bm{z}^{k} + \delta(\bm{y}^{k+1} - \bm{x}^{k}).
              \end{cases}
      \end{aligned}$$
      Introduce the mappings $ R_\beta:\mathbb{R}^n\to \mathbb{R} $ and $ S_\beta:\mathbb{R}^n\to \mathbb{R} $ defined by
      $$
          \setlength{\abovedisplayskip}{8pt}   % 上方间距
          \setlength{\belowdisplayskip}{5pt}   % 下方间距
              R_{\beta}(\bm{z}) = {\rm prox}_{\beta g}(2{\rm prox}_{\beta f}(\bm{z}) - \bm{z}), \quad S_{\beta}(\bm{z}) = \bm{z} + \delta(R_{\beta}(\bm{z})-{\rm prox}_{\beta f}(\bm{z})).
      $$
      It is worth noting that $ R_{\beta}(\bm{z}) $ is the solution to the DR envelope (DRE) \cite{themelis2020douglas} at $ \bm{z} $:
      \begin{equation}\label{Douglas-Rachford envelope}
          \setlength{\abovedisplayskip}{8pt}   % 上方间距
          \setlength{\belowdisplayskip}{5pt}   % 下方间距
          \min_{\bm{y}\in\mathbb{R}^n} \left\langle \nabla f(\bm{x}(\bm{z})), \bm{y} - \bm{x}(\bm{z}) \right\rangle + g(\bm{y}) + \frac{1}{2\beta} \|\bm{y} - \bm{x}(\bm{z})\|^2,
      \end{equation}
      where $ \bm{x}(\bm{z}) = \text{prox}_{\beta f}(\bm{z}) $. The DRS update \eqref{Douglas-Rachford splitting} can be viewed as a fixed-point iteration: $\bm{z}^{k+1} = S_{\beta}(\bm{z}^k).$ Let $ \bm{z}^* $ be the fixed point of $ S_\beta $. Then, the critical point $ \bm{x}^* $ of problem \eqref{Problem:DRS} and $ \bm{z}^* $ satisfy the following relations \cite[Proposition 3]{atenas2024weakly}:
      $$
          \setlength{\abovedisplayskip}{8pt}   % 上方间距
          \setlength{\belowdisplayskip}{5pt}   % 下方间距
          \bm{x}^* = {\rm prox}_{\beta f}(\bm{z}^*) \quad \text{and} \quad -\nabla f(\bm{x}^*) \in \partial g(\bm{x}^*) \Leftrightarrow \bm{z}^* = S_{\beta}(\bm{z}^*).
      $$
      
      Similarly, we can ensure the smoothness of the mapping $ S_{\beta}(\bm{z}) $ around $ \bm{z}^* $ with the help of the active manifold identification property and Lemma \ref{Local smoothness of subproblem mappings}. We then establish the locally R-linear convergence rate of the Anderson-accelerated DRS algorithm.
      
      \begin{theorem}\label{Cov:DRS}
      Consider the optimization problem \eqref{Problem:DRS}. Let $ \bm{x}^* $ be a critical point of $ F $ and $ \bm{z}^* $ a fixed point of $ H(\bm{z}) $ with $ \beta \in(0, \min\{\sigma^{-1},L_f^{-1}\}) $. If $ g $ admits a $ C^3 $-active manifold $ \mathcal{M} $ at $ \bm{x}^* $, $ f $ is $ \mathcal{C}^4 $-smooth near $ \bm{x}^* $, and $ \bm{0} \in \text{rint}(\partial F(\bm{x}^*)) $, then:
      \begin{itemize}
      \item[(i)] $ H $ is $ \mathcal{C}^{2} $-smooth; both $ H $ and $ \nabla H $ are Lipschitz continuous around $ \bm{z}^* $.
      \item[(ii)] Under Assumption \ref{Assumption AA}, if the initial point $ \bm{x}^0 $ is sufficiently close to $ \bm{x}^* $, then applying Anderson acceleration to $ \bm{z} = H(\bm{z}) := S_{\beta}(\bm{z}) $ according to Algorithm \ref{AA fixed-point} leads to convergence of the generated iterates to $ \bm{z}^* $ R-linearly with $ \hat{\gamma} \in (\gamma, 1) $; i.e. \eqref{cov-rate} holds.
      \end{itemize}
      \end{theorem}
      
      \begin{proof}
      Research has established the active manifold identification property and the local smoothness of $ H $. For completeness, we briefly summarize these results; additional details can be found in \cite[Theorem 2]{atenas2024weakly}. In the context of the DRS algorithm, the active manifold identification property manifests as follows: For all $ \bm{z} $ near $ \bm{z}^* $, we have $ \bm{y} = R_{\beta}(\bm{z}) \in \mathcal{M} $. By applying Lemma \ref{Local smoothness of subproblem mappings} to the DRE as defined in \eqref{Douglas-Rachford envelope}, with the condition $ \beta \in (0, \min\{\sigma^{-1},L_f^{-1}\}) $, we deduce that $ R_{\beta} $ is $ \mathcal{C}^2 $-smooth around $ \bm{z}^* $. Consequently, this implies the $ \mathcal{C}^2 $-smoothness of $ S_{\beta} $. Following a similar argument as in Theorem \ref{Convergence for AA prox-type algorithm}, we conclude the proof.
      \end{proof}
      
      \begin{remark}
      Given the established equivalence between DRS and ADMM \cite{eckstein1992douglas,themelis2020douglas}, our analysis also applies to Anderson-accelerated ADMM. We consider the following linearly constrained optimization problem:
      \begin{equation}\label{problem:ADMM}
          \setlength{\abovedisplayskip}{8pt}   % 上方间距
          \setlength{\belowdisplayskip}{5pt}   % 下方间距
      \min\limits _{(\bm{u},\bm{w})\in \mathbb{R}^m\times \mathbb{R}^n} \:  \varphi_1(\bm{u})+\varphi_2(\bm{w})  \quad \text{s.t.} \quad A\bm{u}+B\bm{w}=\bm{b},
      \end{equation}
      where $\varphi_1:\mathbb{R}^m\rightarrow \mathbb{R}\cup \{\infty\}$, $\varphi_2:\mathbb{R}^n\rightarrow \mathbb{R}\cup \{\infty\}$, $A\in\mathbb{R}^{p\times m}$, $B\in\mathbb{R}^{p\times n}$ and $\bm{b}\in \mathbb{R}^p$. 
      The iterative scheme of ADMM is described as follows:
      $$
          \setlength{\abovedisplayskip}{8pt}   % 上方间距
          \setlength{\belowdisplayskip}{5pt}   % 下方间距
          \begin{aligned}\label{ADMM}
              \begin{cases}
                  \bm{u}^{k+1} &= \arg\min L_{\lambda}(\cdot,\bm{w}^{k},\bm{v}^{k}), \\
                  \bm{v}^{k+1} &= \bm{v}^k + \lambda (A\bm{u}^{k+1}+B\bm{w}^k-\bm{b}), \\
                  \bm{w}^{k+1} &= \arg\min L_{\lambda}(\bm{w}^{k+1},\cdot,\bm{v}^{k+1}).
              \end{cases}
          \end{aligned}
      $$
      Here, $\lambda>0$ is a penalty parameter, and the augmented Lagrangian $L_\lambda$ is defined as:
      $$
          \setlength{\abovedisplayskip}{8pt}   % 上方间距
          \setlength{\belowdisplayskip}{5pt}   % 下方间距
          L_{\lambda}(\bm{u},\bm{w},\bm{v}):=\varphi_1(\bm{u})+\varphi_2(\bm{w})+\langle \bm{v},A\bm{u}+B\bm{w}-\bm{b} \rangle + \frac{\lambda}{2}\|A\bm{u}+B\bm{w}-\bm{b}\|^2.
      $$
      For nonconvex problems, \cite{themelis2020douglas} thoroughly analyzes the equivalence between ADMM and DRS using the notion of the \textit{image function}. Through a primal reformulation, it is shown that ADMM for \eqref{problem:ADMM} is essentially equivalent to applying DRS to \eqref{Problem:DRS} with $ f=(A\varphi_1) $ and $ g=(B\varphi_2)(\bm{b}-\cdot) $. For a detailed explanation, readers are encouraged to refer to \cite[Section 5.1]{themelis2020douglas}. The correspondence between DRS variables $(\bm{x}, \bm{y}, \bm{z})$ and ADMM variables $(\bm{u}, \bm{v}, \bm{w})$ is established in \cite[Theorem 5.5]{themelis2020douglas} as follows:
      \begin{equation}\label{Variables correspondence}
          \setlength{\abovedisplayskip}{8pt}   % 上方间距
          \setlength{\belowdisplayskip}{5pt}   % 下方间距
          \bm{x} = A\bm{u}, \quad 
          \bm{y} = \bm{b}-B\bm{w}, \quad 
          \bm{z} = A\bm{u} - \bm{v}/\lambda.
      \end{equation}
      We now proceed to establish the local smoothness of the ADMM mapping $(\bm{u}^k, \bm{v}^k, \bm{w}^k) \\ \overset{\text{ADMM}}\longmapsto (\bm{u}^{k+1}, \bm{v}^{k+1}, \bm{w}^{k+1})$ under the following two conditions:
      
      (i) Problem \eqref{Problem:DRS} with $ f=(A\varphi_1) $ and $ g=(B\varphi_2)(\bm{b}-\cdot) $ satisfies the assumptions in Theorem \ref{Cov:DRS}.
      
      (ii) The correspondence between DRS variables $(\bm{x}, \bm{y}, \bm{z})$ and ADMM variables $(\bm{u}, \bm{v}, \bm{w})$ is one-to-one, meaning $m=n=p$ and both $A$ and $B$ are invertible.
      
      Under the first condition, Theorem \ref{Cov:DRS} establishes the local smoothness of the mappings $R_{\beta}$ and $S_{\beta}$, further demonstrating the local smoothness of the entire mapping $(\bm{x}^k, \bm{y}^k, \bm{z}^k) \overset{\text{DRS}}\longmapsto (\bm{x}^{k+1}, \bm{y}^{k+1}, \bm{z}^{k+1})$ as defined in \eqref{Douglas-Rachford splitting}. The second condition guarantees a one-to-one correspondence between $(\bm{x}, \bm{y}, \bm{z})$ and $(\bm{u}, \bm{v}, \bm{w})$, establishing that the mapping $(\bm{u}^k, \bm{v}^k, \bm{w}^k) \overset{\text{ADMM}}\longmapsto (\bm{u}^{k+1}, \bm{v}^{k+1}, \bm{w}^{k+1})$ is smooth. Using similar reasoning, we can also prove the local convergence rate of Anderson acceleration for $(\bm{u}^k, \bm{v}^k, \bm{w}^k) \overset{\text{ADMM}}\longmapsto (\bm{u}^{k+1}, \bm{v}^{k+1}, \bm{w}^{k+1})$.
      \end{remark}
      
      \section{\texorpdfstring{Anderson-accelerated iteratively reweighted $\ell_1$ algorithm}{Anderson-accelerated iteratively reweighted l1 algorithm}}\label{Sec:AAIRL1}
      In the previous sections, we established the local smoothness of iteration mappings for various algorithms within the context of weakly convex nonsmooth functions. In this section, we shift our focus to a class of nonconvex regularized optimization problems and the corresponding algorithm, the iteratively reweighted $\ell_1$ algorithm (IRL1) \cite{gong2013general,lu2014iterative}. We will demonstrate the local smoothness of the iteration mapping of IRL1 and confirm the local linear convergence rate of Anderson acceleration for this algorithm.

      Consider the following nonconvex regularized optimization problem:
\begin{equation}
    \label{equation1}
    \setlength{\abovedisplayskip}{8pt}   % 上方间距
    \setlength{\belowdisplayskip}{5pt}   % 下方间距
    \min\limits _{\bm{x}\in \mathbb{R}^n} \:F(\bm{x}):=f(\bm{x})+  \lambda \displaystyle\sum^n_{i=1} \phi (\vert x_i\vert),
\end{equation}
where $f: \mathbb{R}^n \to \mathbb{R}$ is Lipschitz continuously differentiable, $\lambda > 0$, and $\phi:\mathbb{R}_+ \to \mathbb{R}_+$ is a nonsmooth regularization function.
%Problem \eqref{equation1} arises from diverse fields including signal processing  \cite{philiastides2006temporal,tropp2006just}, biomedical informatics  \cite{tsuruoka2007learning,liao2007logistic} to modern machine learning \cite{mairal2010online,scardapane2017group}. 
% The primary goal of these problems is to derive solutions with the maximum number of zero elements. Nonconvex and nonsmooth regularization is a prevalent tool for inducing sparsity, serving as an approximation of $\ell_0$ regularization with the common problem form \eqref{equation1}.  
The function $\phi$ takes various forms, such as the EXP approximation \cite{bradley1998feature}, LPN approximation \cite{fazel2003log}, LOG approximation \cite{lobo2007portfolio}, FRA approximation \cite{fazel2003log}, and TAN approximation \cite{candes2008enhancing}. Table \ref{tab1} presents the explicit forms for these cases. $p$ is a hyperparameter that controls sparsity. In this section, we impose the following assumptions on $F$ defined in \eqref{equation1}.  
\begin{table}[htbp]
    \vspace{0pt}
    \renewcommand\arraystretch{1}
    \begin{center}
    \begin{tabular}{@{}llllll@{}}
    \hline
      & EXP \cite{bradley1998feature} & LPN \cite{fazel2003log} & LOG \cite{lobo2007portfolio} & FRA \cite{fazel2003log} & TAN \cite{candes2008enhancing} \\
    \hline
    $p$ & $(0,+\infty)$ & $(0,1)$ & $(0,+\infty)$ & $(0,+\infty)$ & $(0,+\infty)$ \\
    $\phi(|x_i|)$ & $1-e^{-p|x_i|}$ & $|x|^p_i$ & $\log(1+p|x_i|)$ & $\frac{|x_i|}{|x_i|+p}$ & ${\rm arctan}(p|x_i|)$ \\
    \hline
    \end{tabular}
    \end{center}
    \vspace{-7pt}
    \caption{Different functions $\phi(|x_i|)$} \label{tab1}
    \vspace{-17pt}
\end{table}
% \begin{table}[htbp]
%     \vspace{-1pt}
%     \renewcommand\arraystretch{1.5}
%     %\scalebox{1.5}{
%       \begin{center}
%     \begin{tabular}{@{}llll@{}}
%     %\begin{tabular}{@{}lll@{}}
%     \hline
%       & $p$ & $\phi(|x_i|)$  & $\tilde{\omega}_i$  \\
%     \hline
    
%     EXP  \cite{bradley1998feature} & $(0,+\infty)$    & $1-e^{-p|x_i|}$   & $ pe^{-p(\|x_i\|_1)}$  \\

%     LPN  \cite{fazel2003log} & $(0,1)$   &  $|x|^p_i$   & $p(\|x_i\|_1+\tilde{\epsilon}_i)^{p-1}$ \\
        
%     LOG  \cite{lobo2007portfolio} & $(0,+\infty)$    & $\log(1+p|x_i|)$   & $\frac{p}{1+p\|x_i\|_1} $ \\

%     FRA  \cite{fazel2003log}  & $(0,+\infty)$    & $\frac{|x_i|}{|x_i|+p}$   & $\frac{p}{(\|x_i\|_1+p)^2} $  \\

%     TAN  \cite{candes2008enhancing}  & $(0,+\infty)$     & ${\rm arctan}(p|x_i|)$   & $\frac{p}{1+p^2(\|x_i\|_1)^2} $ \\
%     \hline
%     \end{tabular}
%   \end{center}
%    \caption{Different functions $\phi(|x_i|)$ \red{the column of $\tilde \omega_i$ seems unnecessary.}} \label{tab1}
% \end{table}
\begin{assumption}\label{Assumption1} 
    \begin{itemize}
        \item [(i)] $f$ has $L_f$-Lipschitz continuous gradients.

        \item [(ii)]  $\phi$ is smooth, concave, and strictly increasing on $(0, +\infty)$, with $\phi(0) = 0$.
            
        \item [(iii)]   $F$ is non-degenerate in the sense that for any critical point $\bm{x}^*$ of $F$, it holds that $\bm{0} \in \mathop{{\rm rint}}\partial F(\bm{x}^*)$.
    \end{itemize}
\end{assumption}

For problem \eqref{equation1}, we provide a more intuitive derivation using non-degeneracy conditions to demonstrate that IRL1 possesses the property of active manifold identification. For the LPN approximation, the non-degeneracy condition naturally holds because \( \lambda\phi'(0^+) = \lambda p(0^+)^{p-1} = \infty \), which significantly dominates \( \nabla_i f(0^+) \) in the optimality condition. For other sparse approximations, the values of \( \lambda \) and \( p \) are typically adjusted to ensure sparsity, which requires either \( \lambda \) or \( \phi'(0^+) \) to be sufficiently large to satisfy the non-degeneracy condition.

To address the nonsmoothness of the objective function in \eqref{equation1}, a common technique in IRL1 involves adding a perturbation vector $\bm{\epsilon} \in \mathbb{R}^n_{+}$, leading to a continuously differentiable function:
$$
    \setlength{\abovedisplayskip}{8pt}   % 上方间距
    \setlength{\belowdisplayskip}{5pt}   % 下方间距
    F(\bm{x}, \bm{\epsilon}) := f(\bm{x}) + \lambda \sum_{i=1}^n \phi(|x_i| + \epsilon_i).
$$
Under Assumption \ref{Assumption1}, for any $\beta \in (0, L_f^{-1})$ and for any $\bm{x}$ near $\bm{x}^k$, the following holds:
$$
    \setlength{\abovedisplayskip}{8pt}   % 上方间距
    \setlength{\belowdisplayskip}{5pt}   % 下方间距
    F(\bm{x}, \bm{\epsilon}) \leq f(\bm{x}^k) + \nabla f(\bm{x}^k)^T (\bm{x} - \bm{x}^k) + \frac{1}{2\beta} \|\bm{x} - \bm{x}^k\|_2^2 + \lambda \sum_{i=1}^n \phi'(|x_i^k| + \epsilon_i^k)(|x_i| - |x_i^k|).
$$
Consequently, in the $k$-th iteration of IRL1, a convex subproblem that locally approximates $F$ is solved to obtain the new iterate $\bm{x}^{k+1}$:
\begin{equation}\label{equation3}
    \setlength{\abovedisplayskip}{8pt}   % 上方间距
    \setlength{\belowdisplayskip}{5pt}   % 下方间距
    \bm{x}^{k+1} = \mathop{\arg\min}_{\bm{x} \in \mathbb{R}^n} G_k(\bm{x}) := Q_k(\bm{x}) + \lambda \sum_{i=1}^n \omega_i^k |x_i|,
\end{equation}
where $\omega_i^k := \phi'(|x_i^k| + \epsilon_i^k)$ and $Q_k(\bm{x}) := \nabla f(\bm{x}^k)^T \bm{x} + \frac{1}{2\beta} \|\bm{x} - \bm{x}^k\|^2.$
% Denote $W^k:=\mathop{{\rm diag}}(\omega^k_1,...\omega^k_n)$. 
% The ﬁrst-order necessary optimality condition of the subproblem (\ref{equation3}) is given as follows
% \begin{equation}
%     \begin{aligned} \label{equation11}
%         \nabla Q_k(\bm{x}^{k+1}) + \lambda W^k \xi^k = 0,
% \end{aligned}
% \end{equation}
% where $\xi^k \in \partial \|\bm{x}^{k+1}\|_1$. 
The subproblem \eqref{equation3} has a closed-form solution:
\begin{equation}\label{equation6}
    \setlength{\abovedisplayskip}{8pt}   % 上方间距
    \setlength{\belowdisplayskip}{5pt}   % 下方间距
	x^{k+1}_i =
	\begin{cases}
		x^k_i -\beta\nabla_i f(\bm{x}^k) + \beta\lambda \omega^k_i   & \text{if } x^k_i -\beta\nabla_i f(\bm{x}^k) < -\beta\lambda \omega^k_i,\\
  
        x^k_i -\beta\nabla_i f(\bm{x}^k) -\beta\lambda \omega^k_i   & \text{if } x^k_i -\beta\nabla_i f(\bm{x}^k) > \beta\lambda \omega^k_i,\\
		0   & \text{otherwise}.
	\end{cases}
\end{equation}
In the IRL1 algorithm, the choice of $\bm{\epsilon}$ significantly impacts performance. Large $\bm{\epsilon}$ smoothes out overlook many local minimizers. while small $\bm{\epsilon}$ make the subproblem difficult to solve due to bad minimizers. A common strategy is to initialize the algorithm with a relatively large $\bm{\epsilon}^0$ and drive it toward $0$ during successive iterations. Dynamic update schemes for $\bm{\epsilon}$ have been proposed by Wang et al. \cite{wang2021relating} and Lu \cite{lu2014iterative}. % Additionally, Wang et al. \cite{wang2021relating} highlight the locally stable sign property of the sequence $\{\bm{x}^k\}$ generated by IRL1. This indicates that $\textrm{sign}(\bm{x}^k)$ remains unchanged for sufficiently large $k$.

% Additionally, Wang et al. \cite{wang2021relating} highlight the locally stable sign property of the sequence $\{\bm{x}^k\}$ generated by IRL1. This indicates that $\textrm{sign}(\bm{x}^k)$ remains unchanged for sufficiently large $k$.
In this paper, we adopt the update scheme proposed by these authors:
\begin{equation}\label{eps-upd}
    \setlength{\abovedisplayskip}{8pt}   % 上方间距
    \setlength{\belowdisplayskip}{5pt}   % 下方间距
    \bm{\epsilon}^{k+1} = \mu \bm{\epsilon}^{k}
\end{equation}
in IRL1, where $\mu \in (0,1)$ controls the decay rate. By introducing
\begin{equation}\label{H}
    \setlength{\abovedisplayskip}{8pt}   % 上方间距
    \setlength{\belowdisplayskip}{5pt}   % 下方间距
    \begin{aligned}
    \bm{\theta} := \begin{bmatrix}
    \bm{x} \\ 
    \bm{\epsilon}
    \end{bmatrix} \quad \text{and} \quad H: \mathbb{R}^{2n} \to \mathbb{R}^{2n} \text{ defined as } H(\bm{\theta}) := \begin{bmatrix}
                \arg\min_{\bm{y} \in \mathbb{R}^n} G_k(\bm{y})  \\
                \mu \bm{\epsilon}
                \end{bmatrix},
    \end{aligned}
\end{equation}
we can represent the iteration of IRL1, i.e. \eqref{equation3} and \eqref{eps-upd}, as a fixed-point iteration ${\bm \theta}^{k+1} = H(\bm \theta^k).$ %  of the compound variable $\bm \theta:=(\bm{x},\bm{\epsilon})$, that is 
	% $
       % \left[\begin{array}{c}
        %    \bm{x}^{k+1}  \\
        %    \bm{\epsilon}^{k+1}
        %    \end{array}\right]
          %:= 
       % \left[\begin{array}{c}
        %    H_{\bm{x}}(\bm{x}^k,\bm{\epsilon}^k)  \\
         %   H_{\bm{\epsilon}}(\bm{x}^k,\bm{\epsilon}^k)
         %   \end{array}\right]
         %   = 
       % \left[\begin{array}{c}
        %    \arg\min_{\bm{x} \in\mathbb{R}^n} G_k(\bm{x};\bm{x}^k,\bm{\epsilon}^k)  \\
        %    \mu\bm{\bm{\epsilon}}^{k} 
        %    \end{array}\right].
% $ 
Denote $\bm \theta^* = [\bm x^* ; \bm 0].$ For simplicity, we introduce $H(\bm \theta)=[H_1(\bm \theta); H_2(\bm \theta)]$ with mappings $H_1:\mathbb R^{2n}\to \mathbb \R^n$ and $H_2:\R^{2n}\to \R^n$ defined by 
$$
    \setlength{\abovedisplayskip}{8pt}   % 上方间距
    \setlength{\belowdisplayskip}{5pt}   % 下方间距
    H_1(\bm{\theta}) = \arg\min_{\bm{y} \in\mathbb{R}^n} G_k(\bm{y}), \quad H_2(\bm{\theta})= \mu \bm{\epsilon}.
$$
We proceed to analyze the local smoothness of $H$. Denote $\mathcal{A}^* := \{i \mid x^*_i = 0\}$ and $\mathcal{I}^* := \{i \mid x^*_i \neq 0\}$. A notable property of IRL1 for \eqref{equation1} is the locally stable sign property of the iterates $\{\bm{x}^k\}$, meaning that $\textrm{sign}(\bm{x}^k)$ remains unchanged for sufficiently large $k$ \cite{wang2021relating}. Towards the end of the iteration, IRL1 is equivalent to solving a smooth problem in the reduced space $\mathbb{R}^{\mathcal{I}^*}$, a property known as model identification. We generalize this concept to the general nonconvex regularized problem \eqref{equation1} by leveraging the active manifold identification property, which can be established based on the non-degeneracy condition required by Assumption \ref{Assumption1}. We then infer the local smoothness of $H$.

\begin{lemma}[Active manifold identification and local smoothness]\label{Lemma3} 
    Let Assumption \ref{Assumption1} hold. There exists a neighborhood $\mathcal{B}_{\bm{\theta^*}}(\rho)$ of $\bm{\theta}^*$ such that for all $\bm{\theta} \in \mathcal{B}_{\bm{\theta^*}}(\rho)$, the following statements hold:
	\begin{itemize}
        \item [{\rm (i)}] $H(\bm{\theta}) \in \mathcal{M}$, where $\mathcal{M} := \left\{\bm{\theta} = [\bm{x}; \bm{\epsilon}] \, \big| \, \textrm{sign}(\bm{x}) = \textrm{sign}(\bm{x}^*), \bm{\epsilon} \in \mathbb{R}^n_+\right\}.$
        \item [{\rm (ii)}] $H(\bm{\theta})$ is continuously differentiable. 
    \end{itemize}
\end{lemma}

\begin{proof}
(i) It is straightforward to observe that $H_2(\bm{\theta}) = \mu \bm{\epsilon} \in \mathbb{R}^n_+$. 

To prove that $\textrm{sign}(H_1(\bm{\theta})) = \textrm{sign}(\bm{x}^*)$, consider $i \in \mathcal{A}^*$. The non-degeneracy condition $ -\nabla_i f(\bm{x}^*) \in \mathrm{rint}(\lambda \partial \phi(x^*_i))$ and the symmetry of $\phi$ imply that $|\nabla_i f(\bm{x}^*)| < |\lambda \omega(\theta^*_i)|$. Since $x^*_i = 0$, we have $|x^*_i - \beta \nabla_i f(\bm{x}^*)| < |\beta \lambda \omega(\theta^*_i)|$. Combining this with the continuity of $F$, we deduce that there exists a sufficiently small constant $\rho > 0$ such that for all $\bm{\theta} \in \mathcal{B}_{\bm{\theta^*}}(\rho)$, $|x_i - \beta \nabla_i f(\bm{x})| < |\beta \lambda \omega(\theta_i)|$. Therefore, $[H_1(\bm{\theta})]_i = 0$, which is also evident from \eqref{equation6}. Then we deduce that $\textrm{sign}([H_1(\bm{\theta})]_i) = \textrm{sign}(x^*_i)$. For $i \in \mathcal{I}^*$, the continuity of $H$ implies that there exists a sufficiently small constant $\rho > 0$ such that for all $\bm{\theta} \in \mathcal{B}_{\bm{\theta^*}}(\rho)$, it holds that $\textrm{sign}(x_i) = \textrm{sign}(x_i^*)$ and $\textrm{sign}([H_1(\bm{\theta})]_i) = \textrm{sign}(x_i^*)$.

(ii) Given \( \bm{\theta} \in \mathcal{B}_{\bm{\theta^*}}(\rho) \) and \( \textrm{sign}([H(\bm{\theta})]_i) = \textrm{sign}(x^*_i) \), it follows that \( H(\bm{\theta}) \) is continuously differentiable in \( \mathcal{B}_{\bm{\theta^*}}(\rho) \).
\end{proof}

In the context of IRL1, we provide a specific analysis demonstrating that $H$ is both a contraction and Lipschitz continuously differentiable. Let $\Phi := \sum_{i=1}^n \phi$. We introduce the following assumption for the subsequent analysis.

\begin{assumption}\label{Assumption3}
    Let $\bm{\theta}^* = [\bm{x}^*; 0]$ be the fixed point of $H$, and let $\mathcal{B}_{\bm{\theta^*}}(\rho)$ denote a neighborhood around $\bm{\theta}^*$. For all $\bm{\theta} \in \mathcal{B}_{\bm{\theta^*}}(\rho)$, the following conditions hold:
    \begin{itemize}
        \item[{\rm (i)}] $[\nabla^2_{\bm{xx}}F(\bm{\theta})]_{\mathcal{I}^*} \succeq \kappa I$ for some $\kappa > 0$, and $[\nabla^2_{\bm{xx}}F(\bm{\theta})]_{\mathcal{I}^*}$ is Lipschitz continuous.
        \item[{\rm (ii)}] $[\nabla^2_{\bm{x\epsilon}} \Phi (\bm{\theta})]_{\mathcal{I}^*}$ is Lipschitz continuous with a constant $L_w > 0$.
    \end{itemize}
\end{assumption}

In the following theorem, we show that $H(\bm{\theta})$ is a contraction and Lipschitz continuously differentiable when $\bm{\theta}$ is near $\bm{\theta}^*$, and we establish the local convergence rate of the Anderson-accelerated IRL1.

\begin{theorem}%[Lipschitz and contraction]
\label{Lemma4} Consider the problem \eqref{equation1}. Let Assumptions \ref{Assumption1} and \ref{Assumption3} hold, $\bm{x}^*$ be a critical point of $F$, and $H$ be defined in \eqref{H} with $\mu \in (0, 1 - \sqrt{\lambda^2 L^2_{\omega}/(\kappa^2 + \lambda^2 L^2_{\omega})})$ and $\beta \in (\underline{\beta}, \overline{\beta})$, where $\underline{\beta}, \overline{\beta}$ are the two roots of 
\begin{equation}\label{Bound of beta}
    \setlength{\abovedisplayskip}{8pt}   % 上方间距
    \setlength{\belowdisplayskip}{5pt}   % 下方间距
    (\kappa^2 + \lambda^2 L^2_{\omega}) \beta^2 - 2\kappa \beta + 2\mu - \mu^2 = 0.
\end{equation}
Then the following statements hold:
\begin{itemize}
    \item[{\rm (i)}] $H(\bm{\theta})$ is Lipschitz continuously differentiable and a contraction for all $\bm{\theta} \in \mathcal{B}_{\bm{\theta^*}}(\rho)$, i.e., there exists $\gamma \in (0, 1)$ such that
    \begin{equation}\label{equation22}
        \setlength{\abovedisplayskip}{8pt}   % 上方间距
        \setlength{\belowdisplayskip}{5pt}   % 下方间距
        \|H(\bm{\theta}_1) - H(\bm{\theta}_2)\| \leq \gamma \|\bm{\theta}_1 - \bm{\theta}_2\|, \quad \forall \bm{\theta}_1, \bm{\theta}_2 \in \mathcal{B}_{\bm{\theta^*}}(\rho).
    \end{equation}
    \item[{\rm (ii)}] If the initial point $\bm{\theta}^0$ is sufficiently close to $\bm{\theta}^*$ and Assumption \ref{Assumption AA}(ii) holds, when applying Anderson acceleration to $H$ defined in \eqref{H} following Algorithm \ref{AA fixed-point}, the generated iterates converge to $\bm{\theta}^*$ R-linearly with $\hat{\gamma} \in (\gamma, 1)$, i.e., \eqref{cov-rate} holds.
\end{itemize}
\end{theorem}
\begin{proof} (i) First, we can derive that 
\begin{equation}\label{nabla H (theta)}
    \setlength{\abovedisplayskip}{8pt}   % 上方间距
    \setlength{\belowdisplayskip}{5pt}   % 下方间距
    \begin{aligned}
        \nabla H = \left[\begin{array}{c|c}
        \begin{matrix}
        [\nabla_{\bm{x}} H_1]_{\mathcal{I}^*} \\
        [\nabla_{\bm{x}} H_1]_{\mathcal{A}^*} \\
        \end{matrix} & \nabla_{\bm{x}} H_2 \\
        \hline
        \begin{matrix}
        [\nabla_{\bm{\epsilon}} H_1]_{\mathcal{I}^*} \\
        [\nabla_{\bm{\epsilon}} H_1]_{\mathcal{A}^*} \\
        \end{matrix} & \nabla_{\bm{\epsilon}} H_2 \\
        \end{array}\right]^T
         = \left[\begin{array}{c|c}
        \begin{array}{c}
        I - \beta[\nabla^2_{\bm{xx}}F]_{\mathcal{I}^*} \\
        0 \\
        \end{array} & 0 \\
        \hline
        \begin{array}{c}
        -\beta\lambda[\nabla^2_{\bm{x\epsilon}} \Phi ]_{\mathcal{I}^*}\\
        0 \\
        \end{array} & \mu I \\
        \end{array}\right]^T.
    \end{aligned}
\end{equation}
From the Lipschitz continuity of $[\nabla^2_{\bm{xx}}F]_{\mathcal{I}^*}$ and $[\nabla^2_{\bm{x\epsilon}} \Phi ]_{\mathcal{I}^*}$ as stated in Assumption \ref{Assumption3}, we can deduce that $\nabla H$ is Lipschitz continuous. Additionally, with 
$\|[\nabla^2_{\bm{xx}}F]_{\mathcal{I}^*}  \| \geq \kappa $ and $\|[\nabla^2_{\bm{x\epsilon}} \Phi ]_{\mathcal{I}^*}  \| \leq L_{\omega}$, we further obtain 
\begin{equation}\label{Bound of H1}
    \setlength{\abovedisplayskip}{8pt}   % 上方间距
    \setlength{\belowdisplayskip}{5pt}   % 下方间距
    \|\nabla_{\bm{x}} H_1\| \leq 1 - \beta \kappa \quad \mbox{and} \quad \|\nabla_{\bm{\epsilon}} H_1\| \leq \beta \lambda L_{\omega}.
\end{equation}
From \eqref{nabla H (theta)} we have 
\begin{equation}
    \setlength{\abovedisplayskip}{8pt}   % 上方间距
    \setlength{\belowdisplayskip}{5pt}   % 下方间距
    \begin{aligned} \label{equation11}
        \|\nabla H\| &\leq \|\nabla H_1\| + \|\nabla H_2\| = \sqrt{\|\nabla H_1[\nabla H_1]^T\|} + \|\nabla H_2\|  \\
        &= \sqrt{\|\nabla_{\bm{x}} H_1[\nabla_{\bm{x}} H_1]^T + \nabla_{\bm{\epsilon}} H_1[\nabla_{\bm{\epsilon}} H_1]^T\| } + \|\nabla H_2\| \\
        &\leq \sqrt{\|\nabla_{\bm{x}} H_1\|^2 + \|\nabla_{\bm{\epsilon}} H_1\|^2} + \|\nabla H_2\| \leq \sqrt{(1 - \beta \kappa)^2 + \beta^2 \lambda^2 L_{\omega}^2} + \mu,
    \end{aligned}
\end{equation}
where the first and second inequalities follow from the triangle inequality, and the third inequality follows from \eqref{Bound of H1}. Note that $\mu \in (0, 1 - \sqrt{\lambda^2 L^2_{\omega}/(\kappa^2 + \lambda^2 L^2_{\omega})})$ guarantees the existence of roots of \eqref{Bound of H1}. Combining $\beta \in (\underline{\beta}, \overline{\beta})$ and \eqref{Bound of beta}, we have 
$$
\setlength{\abovedisplayskip}{8pt}   % 上方间距
\setlength{\belowdisplayskip}{5pt}   % 下方间距
    \sqrt{(1 - \beta \kappa)^2 + \beta^2 \lambda^2 L_{\omega}^2} + \mu < 1,
$$
which, together with \eqref{equation11}, implies that $\|\nabla H\| < 1$. That is, $H$ is a contraction for all $\bm{\theta} \in \mathcal{B}_{\bm{\theta^*}}(\rho)$, and \eqref{equation22} holds.

(ii) With the contractivity and Lipschitz continuous differentiability of $H$, applying Theorem \ref{Convergence AA} enables us to establish the local convergence rate.
\end{proof}
% %\vspace{-10pt}
\section{Numerical experiments}
In this section, we demonstrate the performance of Anderson-accelerated nonsmooth optimization algorithms using several well-known examples from signal processing and machine learning. All experiments are implemented in MATLAB 2022a on a 64-bit laptop equipped with an Intel Core i7-1165G7 processor (2.80GHz) and 16GB of RAM. We conduct experiments using both synthetic and real-world datasets, with the initial point $\bm{x}^0$ drawn from a standard Gaussian distribution. For each experiment, we plot the residual norm $\|\bm{r}^k\|$ at each iteration $k$ to validate our theoretical results. Since the additional computational cost of Anderson Acceleration (AA) per iteration is less than 10\%, the runtime plots closely resemble the iteration count plots, we omit the runtime plots in the following experiments.
%\vspace{-0pt}
% \subsection{Synthetic Datasets}
% \subsubsection{Sparse Recovery Problem with PGA}

\subsection{ {PGA for sparse signal recovery}}
% %\vspace{-5pt}
We begin with the classic Lasso problem in signal processing \cite{tibshirani1996regression}:
$$\setlength{\abovedisplayskip}{8pt}   % 上方间距
\setlength{\belowdisplayskip}{5pt}   % 下方间距
\min\limits _{\bm{x}\in \mathbb{R}^N} \:F(\bm{x}):=  \frac{1}{2}\|A\bm{x}-\bm{y}\|^2 +  \lambda \|\bm{x}\|_1,$$
where $A \in \mathbb{R}^{M \times N}$ and $\bm{y} \in \mathbb{R}^{M}$. The objective is to recover a sparse signal $\bm{x}$ from $M$ observations, with $M \ll N$. We compare the performance of classical ISTA, FISTA, and Anderson-accelerated ISTA (AAISTA). The entries of matrix $A$ are generated randomly and independently from a standard Gaussian distribution, and $A$ is orthonormalized along its rows. The true signal $\bm{x}_{\mathop{{\rm true}}}$ is created by randomly selecting $N/10$ elements from an $N$-dimensional zero vector and setting them to ±1. The observation vector is defined as $\bm{y} = A\bm{x}_{\mathop{{\rm true}}} + \bm{\epsilon}$, where $\bm{\epsilon} \in \mathbb{R}^N$ follows a Gaussian distribution with mean 0 and variance $10^{-4}$. We set $\lambda = 0.01$ and $M = 15$. Figure \ref{fig:Lasso-PGA} shows the average performance over 50 random experiments. It can be observed that the AAISTA curve exhibits a significant change toward the end of the iterations, transitioning from a sublinear to a linear, or even superlinear, convergence rate. This behavior suggests that active manifold identification has occurred, significantly accelerating the algorithm. This observation corroborates our theoretical predictions.

%Notably, AAISTA even seems to exhibit superlinear speed, which might occur when the optimal solution is close to the subspace spanned by past iterations.
%\vspace{-0pt}
\begin{figure}[htbp]
\centering
\subfloat[$(200,1000)$]{\label{fig:Lasso-PGA a}\includegraphics[width=0.25\textwidth]{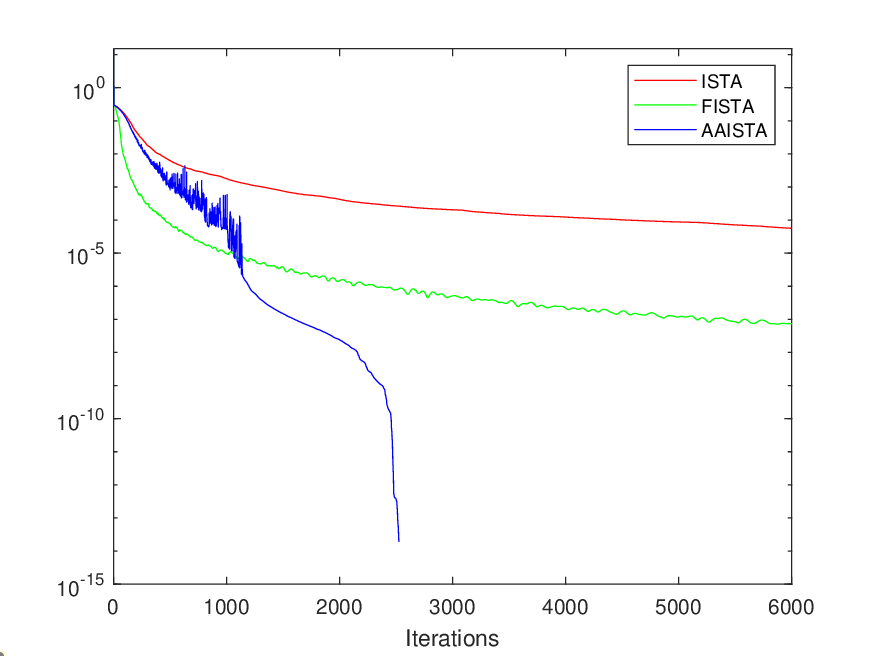}}
\subfloat[$(400,2000)$]{\label{fig:Lasso-PGA b}\includegraphics[width=0.25\textwidth]{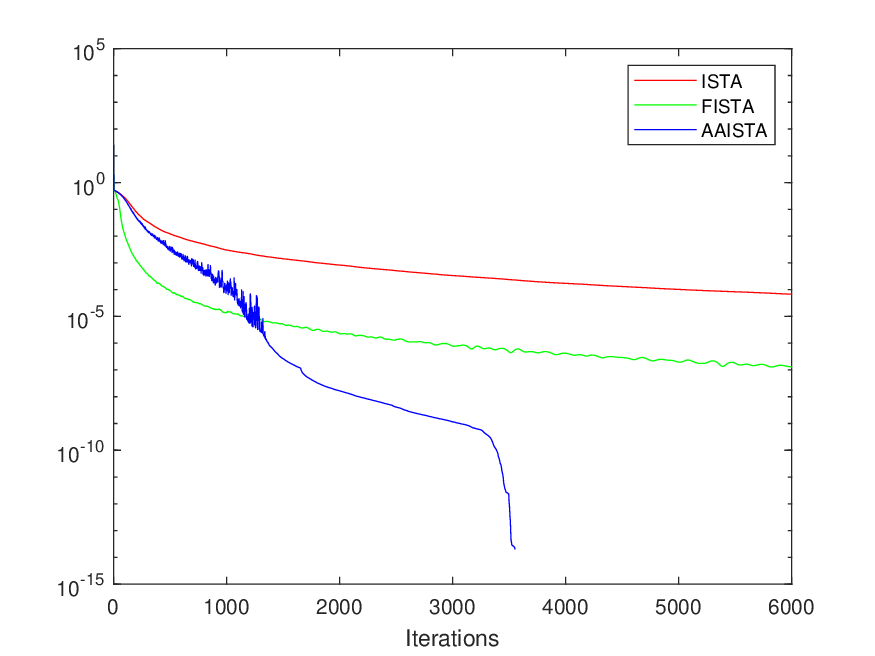}}
\subfloat[$(600,3000)$]{\label{fig:Lasso-PGA C}\includegraphics[width=0.25\textwidth]{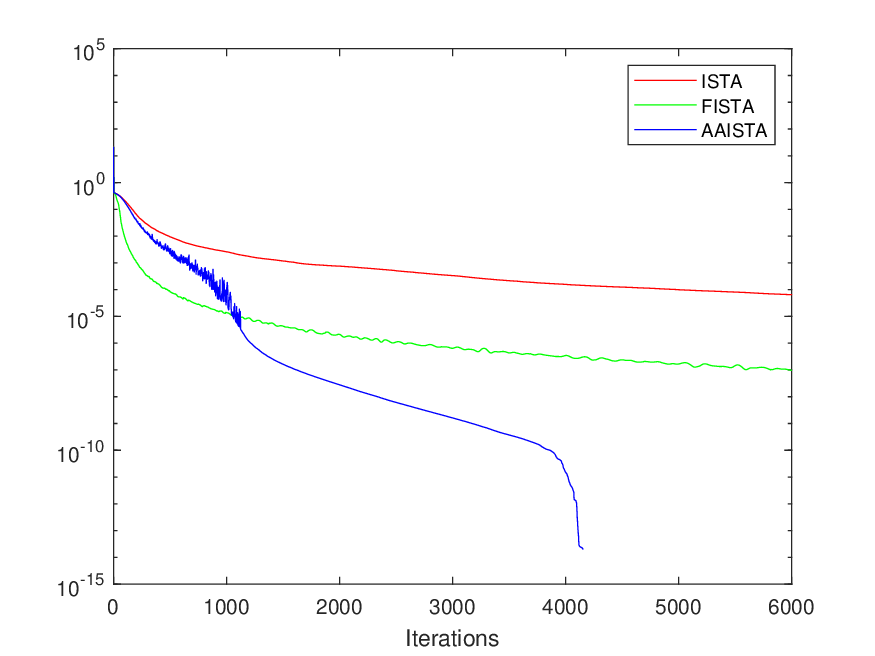}}
%\vspace{-15pt}
\caption{Comparison of $\|\bm{r}^k\|$ for ISTA, FISTA, AAISTA under different $(M,N)$.} 
\label{fig:Lasso-PGA}
%\vspace{-10pt}
\end{figure}
% %\vspace{-10pt}
% \subsection{Real-world datasets}
\subsection{ {PCD, DRS, and IRL1 for classification problems}}
%\vspace{-5pt}
In this section, we compare the performance of the PCD, DRS, and IRL1 algorithms with their respective Anderson-accelerated versions—AAPCD, AADRS, and AAIRL1—across several classification models using five real-world datasets: {\it mushrooms}, {\it colon-cancer}, {\it ijcnn1}, {\it w8a}, and {\it a9a}. These datasets consist of binary classification problems from the LIBSVM repository \cite{chang2011libsvm}. A summary of the datasets is provided in Table \ref{tab:datasets}. Let $\bm{a}_i \in \mathbb{R}^N$ denote the training samples, and $A = [\bm{a}_1, \ldots, \bm{a}_M]^T \in \mathbb{R}^{M \times N}$ represent the matrix of training samples, while $\bm{y} \in \{-1, 1\}^M$ denotes the corresponding labels.
\begin{table}[htbp]
    \renewcommand\arraystretch{1}
    %\scalebox{1.5}{
      \begin{center}
    \begin{tabular}{@{}llllll@{}}
    %\begin{tabular}{@{}lll@{}}
    \hline
    Datasets& mushrooms &  colon-cancer &ijcnn1& a9a & w8a   \\
    \hline
    Samples $M$& 8124 & 62 &49990 & 32561& 49749   \\
    Features $N$&112  & 2000 & 22 & 123 & 300     \\
    \hline
    \end{tabular}
  \end{center}
  %\vspace{-2pt}
    \caption{Characteristics of the datasets} \label{tab:datasets}
    %\vspace{-10pt}
\end{table}
During the experiments, we observed instances of non-convergence when using Anderson Acceleration (abbreviated as AA). To ensure convergence, we incorporated a function/residual descent check to confirm that AA iterates converge to the desired local region. In each experiment, we tested three different memory sizes for AA, $m = 5, 10, 15$, denoted as AA$(m)$. The results indicate that AA demonstrates significant effectiveness, particularly in local regions for real-world tasks. It is also noteworthy that the optimal memory size $m$ may vary depending on the specific problem.
% \vspace{-5pt}
% \subsubsection{Support Vector Machine with PCD}

{\bf  {PCD for support vector machines}} 
In this subsection, we focus on the PCD method for solving the soft margin Support Vector Machine (SVM), which can be tackled through the following dual optimization problem:
$$\setlength{\abovedisplayskip}{8pt}   % 上方间距
\setlength{\belowdisplayskip}{5pt}   % 下方间距
\min\limits _{\bm{x}\in \mathbb{R}^M} \:F(\bm{x}) :=  \frac{1}{2} \|(\bm{y} \odot A)^T \bm{x}\|^2 -\sum^M_{i=1} x_i \quad \text{s.t.} \quad 0 \le \bm{x} \le C,$$
where $\bm{y} \odot A$ denotes the element-wise multiplication between $\bm{y}$ and $A$. This problem can be reformulated in the form \eqref{Problem coordinate descent algorithm} by defining $f(\bm{x}) = F(\bm{x})$ and $g_i(\bm{x}) = \chi_{[0,C]}(x_i)$ for $i = 1, \dots, M$. Soft margin SVM extends the traditional SVM by allowing some misclassifications, with the parameter $C$ controlling the tolerance for misclassification. In our experiments, we set $C = 100$. The PCD update for this problem is given by:
$x_i = \text{prox}_{\beta\chi_{[0,C]}}(x_i - \beta \left( [(\bm{y} \odot A)(\bm{y} \odot A)^T]_{i,:} \bm{x} - 1 \right)),$
where $\beta = 1/L$ and $L = \max_{i=1, \dots, M} \|(\bm{y} \odot A)_{i,:}\|^2$. Given the large number of samples in these datasets, PCD requires a significant number of iterations to compute $x_i$ for each sample. Thus, we randomly selected 2000 samples from each dataset: {\it mushrooms}, {\it ijcnn1}, {\it w8a}, and {\it a9a} for the experiments. Figure \ref{fig:SVM-PCD} compares the performance of AAPCD and the classical PCD algorithm. As shown, the AAPCD curve experiences a sharp decline in the final stages, perfectly aligning with our theoretical predictions. In some challenging experiments, the AA curve closely follows the standard algorithm during the initial iterations, indicating that the AA step is seldom activated early on, leading to an unaccelerated algorithm. This also highlights the global instability of the AA method.

\begin{figure}[htbp]
\centering
\subfloat[mushrooms]{\label{fig:SVM-PCD a}\includegraphics[width=0.25\textwidth]{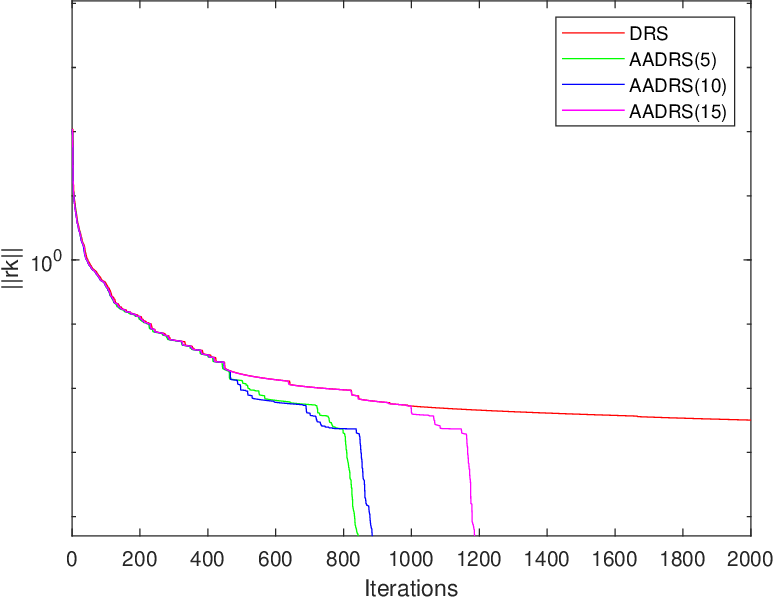}}
\subfloat[colon-cancer]{\label{fig:SVM-PCD b}\includegraphics[width=0.25\textwidth]{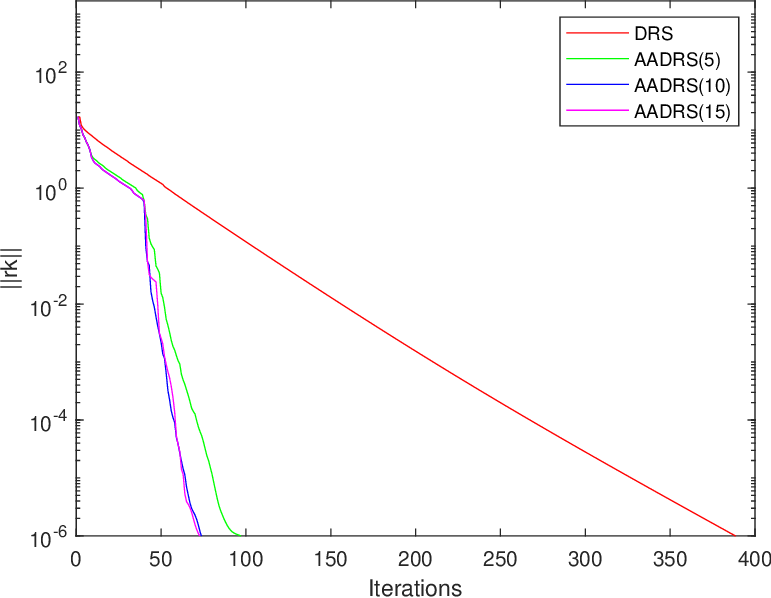}}
\subfloat[ijcnn1]{\label{fig:SVM-PCD C}\includegraphics[width=0.25\textwidth]{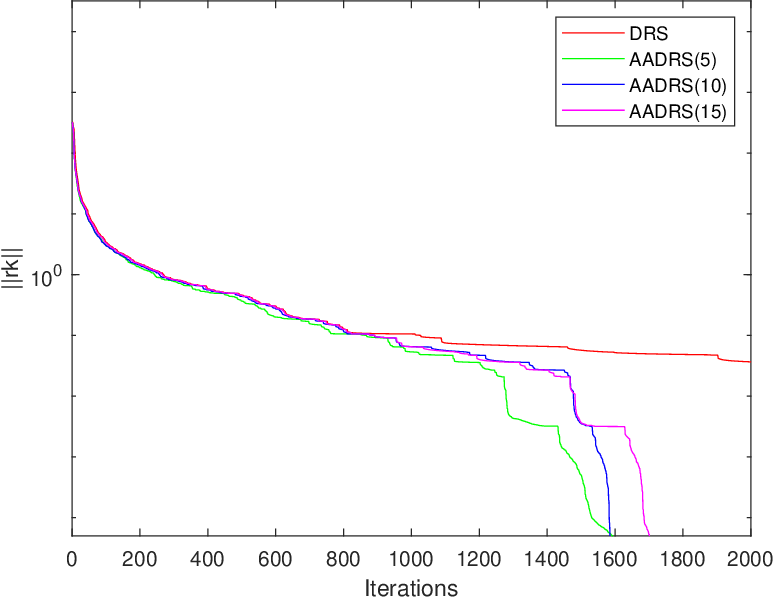}} \\
%\vspace{-0pt}
\subfloat[a9a]{\label{fig:SVM-PCD d}\includegraphics[width=0.25\textwidth]{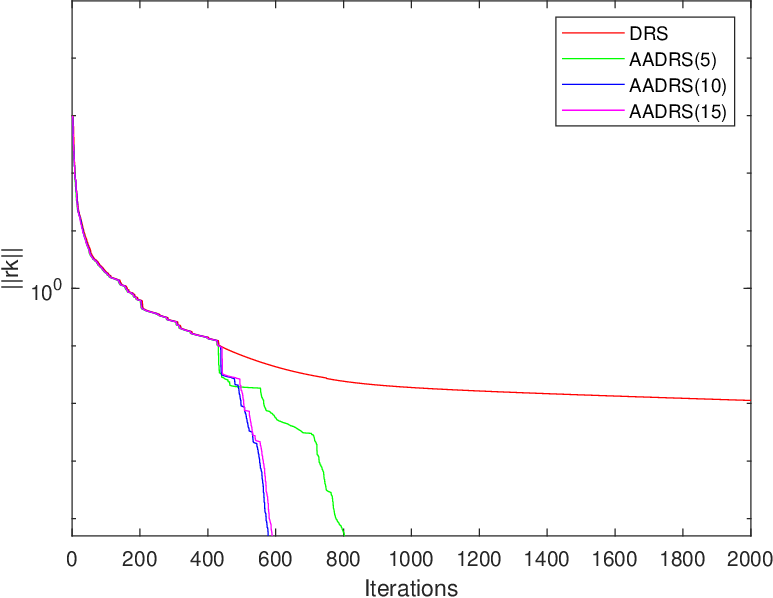}}
\subfloat[w8a]{\label{fig:SVM-PCD e}\includegraphics[width=0.25\textwidth]{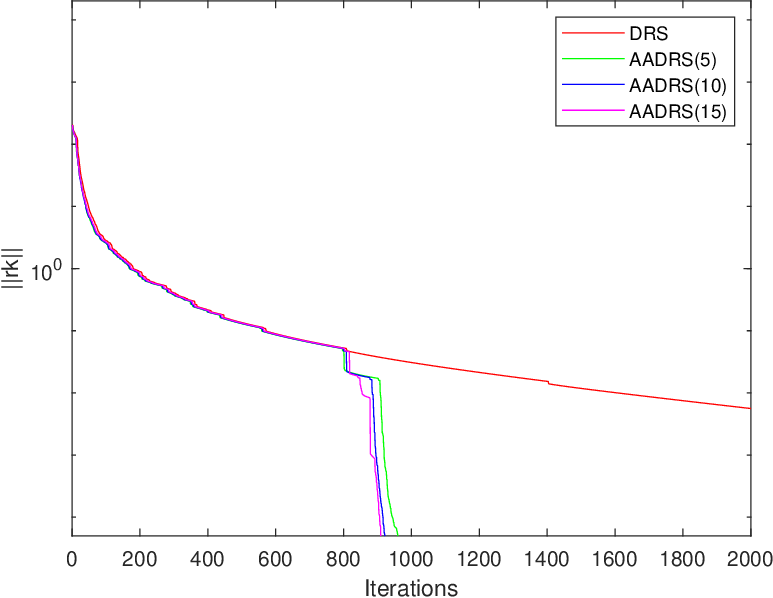}}
%\vspace{-10pt}
\caption{Comparison of $\|\bm{r}^k\|$ for PCD and AAPCD(m).} 
\label{fig:SVM-PCD}
%\vspace{-10pt}
\end{figure}
% \vspace{-10pt}
% \subsubsection{Nonnegative least squares with DRS}
{\bf  {DRS for nonnegative least squares}}
In this section, we explore the DRS method for solving the nonnegative least squares problem:
$$
\setlength{\abovedisplayskip}{8pt}   % 上方间距
\setlength{\belowdisplayskip}{5pt}   % 下方间距
    \min_{\bm{x}\in \mathbb{R}^N} \: F(\bm{x}) := \frac{1}{2M} \|A\bm{x} - \bm{y}\|^2 + \lambda \|\bm{x}\|^2 \quad \text{s.t.} \:\bm{x} \ge 0,
$$
which is a common step in many nonnegative matrix factorization algorithms. This problem can be reformulated as \eqref{Problem:DRS} by defining $f(\bm{x}) = F(\bm{x})$ and $g(\bm{x}) = \chi_{[0,+\infty]}(\bm{x})$. We set $\lambda = 0.001$ and $\beta = 1/L$, where $L = \|A\|^2/M$. The computation of $\text{prox}_{\beta g}$ is straightforward and given by $\max(0, \bm{x})$. For $\text{prox}_{\beta f}$, 
%$$\mbox{prox}_{\beta f}(x) = \arg\min_{t\in\mathbb{R}^N} \|\left[\begin{array}{c}
 %           A \\ \lambda I \\ \frac{1}{\sqrt{2\beta}} I
 %           \end{array}\right] \bm{t} - \left[\begin{array}{c}
%           \bm{y} \\ \bm{0} \\ \frac{1}{\sqrt{2\beta}} \bm{x}
 %           \end{array}\right] \|^2.$$
 it actually corresponds to a least-square problem, which we solve using LSQR, an efficient conjugate gradient method for sparse least-squares problems.
Figure \ref{fig:NLS-DRS} shows the residual curves from our experiments. AA demonstrates significant acceleration. For instance, in Figure \ref{fig:NLS-DRS}(c), AA quickly identifies a high-precision solution, especially when handling ill-conditioned datasets.

\begin{figure}[htbp]
\centering
\subfloat[mushrooms]{\label{fig:NLS-DRS a}\includegraphics[width=0.25\textwidth]{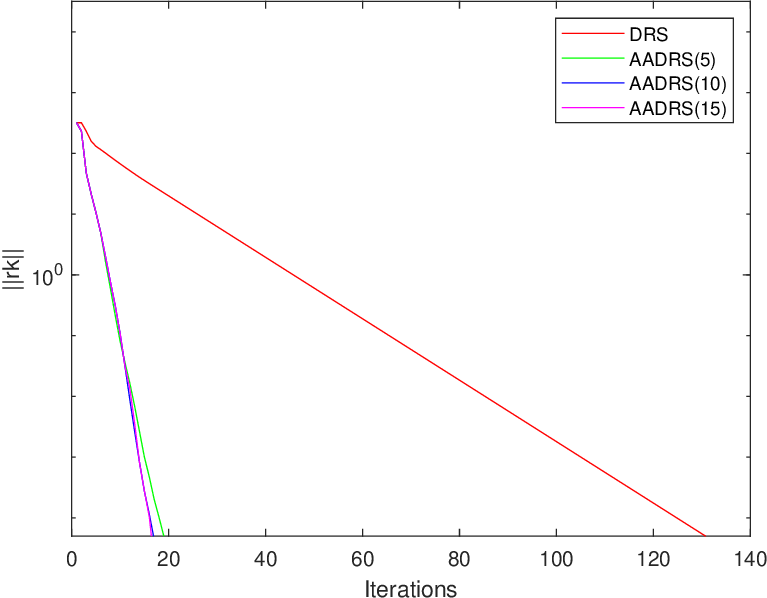}}
\subfloat[colon-cancer]{\label{fig:NLS-DRS b}\includegraphics[width=0.25\textwidth]{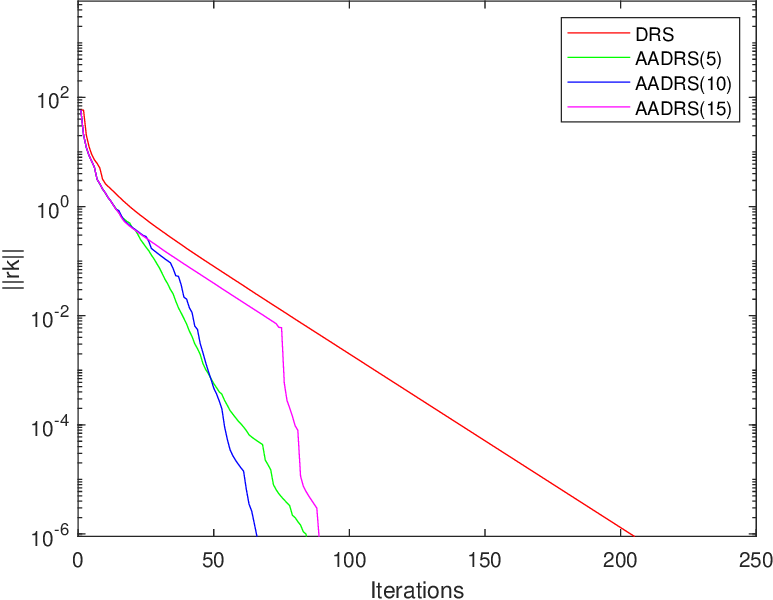}}
\subfloat[ijcnn1]{\label{fig:NLS-DRS C}\includegraphics[width=0.25\textwidth]{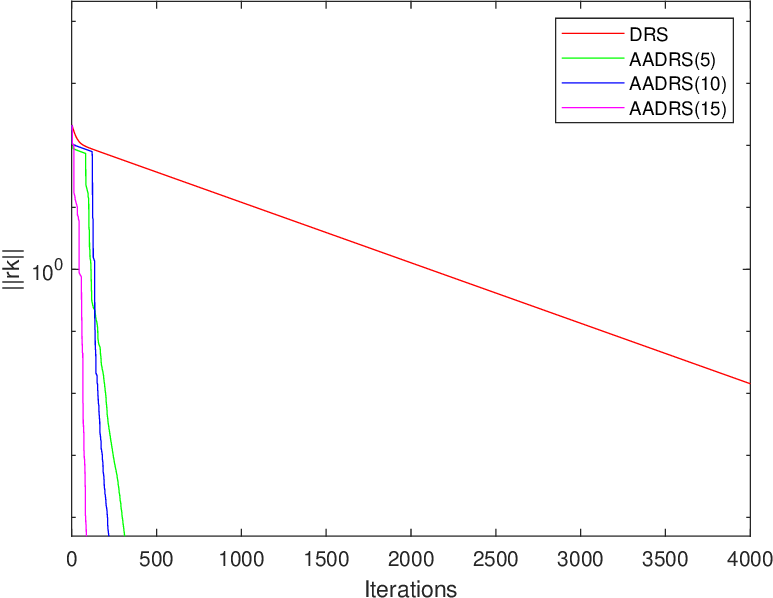}} \\
%\vspace{-10pt}
\subfloat[a9a]{\label{fig:NLS-DRS d}\includegraphics[width=0.25\textwidth]{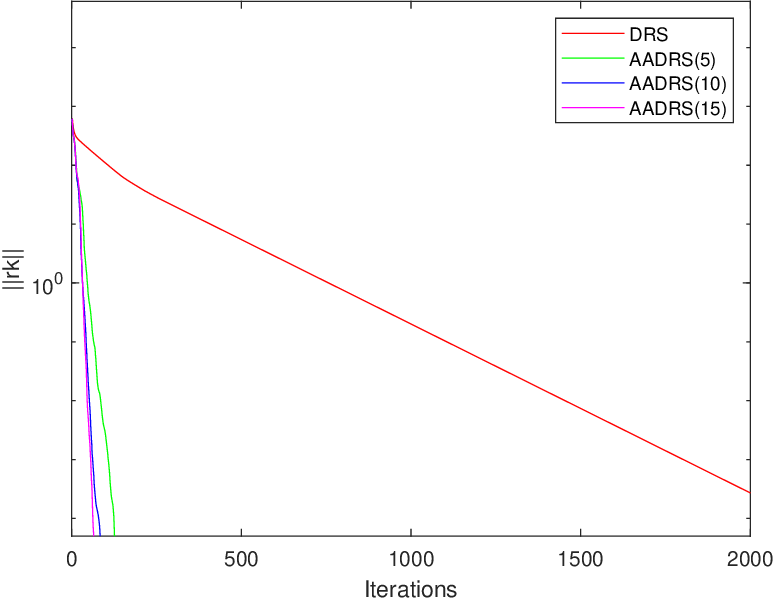}}
\subfloat[w8a]{\label{fig:NLS-DRS e}\includegraphics[width=0.25\textwidth]{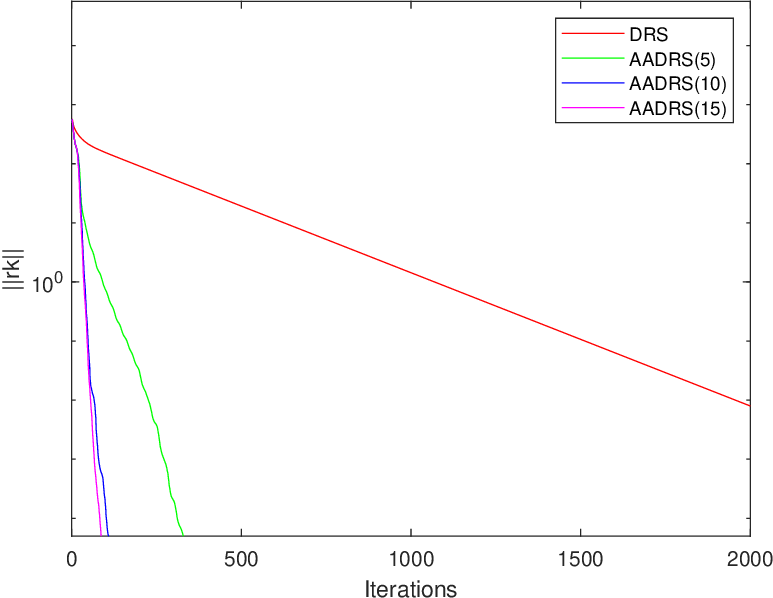}}
%\vspace{-15pt}
\caption{ {Comparison of $\|\bm{r}^k\|$ for DRS and AADRS(m).}} 
\label{fig:NLS-DRS}
%\vspace{-10pt}
\end{figure}
%\vspace{-5pt}

% \subsubsection{Sparse Logistic Regression with IRL1}
{\bf  {IRL1 for sparse logistic regression}}
In this subsection, we report the results of using the IRL1 method to solve the sparse logistic regression problem:
$$\setlength{\abovedisplayskip}{8pt}   % 上方间距
\setlength{\belowdisplayskip}{5pt}   % 下方间距
\min\limits_{\bm{x}\in \mathbb{R}^N} \: F(\bm{x}) :=  \frac{1}{M} \sum^M_{i=1} \log(1 + \exp(-y_i \bm{a}_i^T \bm{x})) +  \lambda \|\bm{x}\|_p^p,$$
where we set $p = 0.75$, $\lambda = 0.001$, $\bm{\epsilon}^0 = \bm{1}$, $\mu = 0.9$, and $\beta = 1/L$, with $L = \|A\|^2/4M$. The results are presented in Figure \ref{fig:SLR-IRL1}. Anderson Acceleration provides a speedup of approximately 5-10 times, with most of the improvement occurring in the final stages of descent. In Figures \ref{fig:SLR-IRL1} (a) and (e), it appears that AA identifies the solution in a finite number of steps, possibly because the solution lies in the subspace spanned by previous iterations.

\begin{figure}[htbp]
\centering
\subfloat[mushrooms]{\label{fig:SLR-IRL1 a}\includegraphics[width=0.25\textwidth]{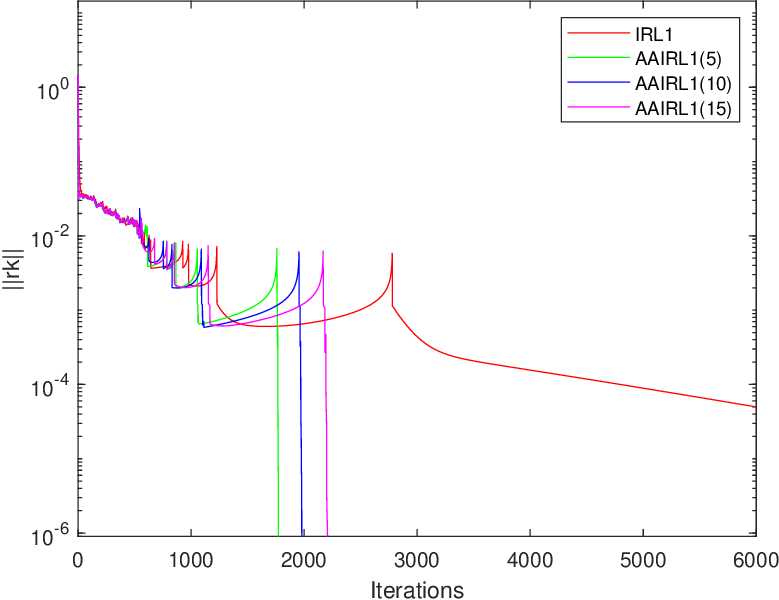}}
\subfloat[colon-cancer]{\label{fig:SLR-IRL1 b}\includegraphics[width=0.25\textwidth]{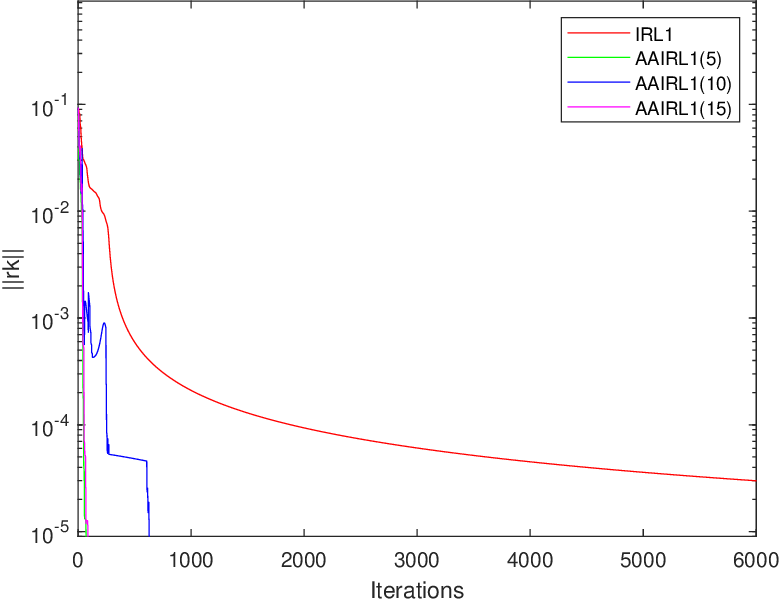}}
\subfloat[ijcnn1]{\label{fig:SLR-IRL1 C}\includegraphics[width=0.25\textwidth]{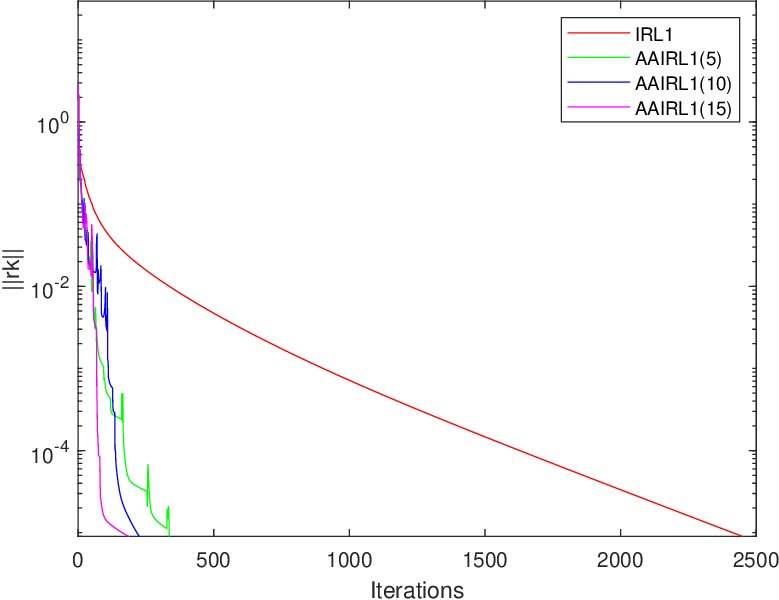}} \\
%\vspace{-10pt}
\subfloat[a9a]{\label{fig:SLR-IRL1 d}\includegraphics[width=0.25\textwidth]{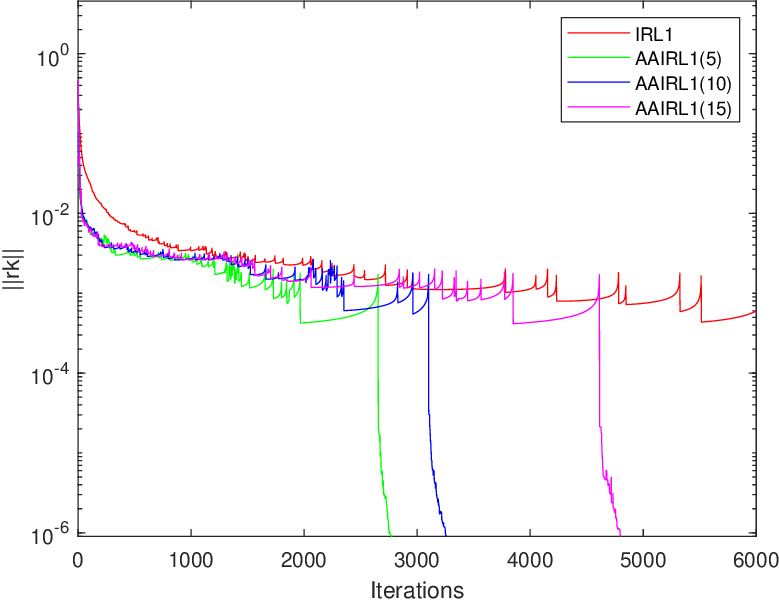}}
\subfloat[w8a]{\label{fig:SLR-IRL1 e}\includegraphics[width=0.25\textwidth]{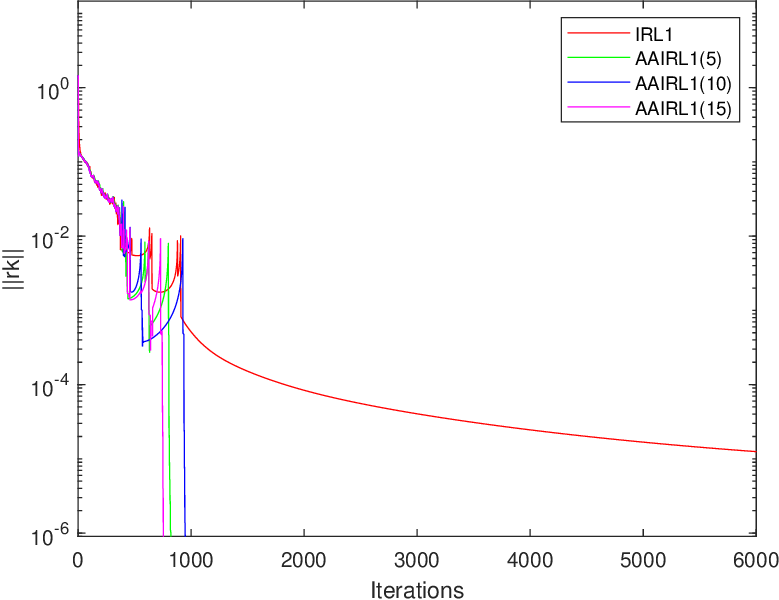}}
%\vspace{-15pt}
\caption{Comparison of $\|\bm{r}^k\|$ for IRL1 and AAIRL1(m).} 
\label{fig:SLR-IRL1}
%\vspace{-10pt}
\end{figure}

%\vspace{-3pt}
\section{Conclusions}
In this paper, we have analyzed the local convergence rate of Anderson acceleration for nonsmooth optimization algorithms that exhibit the active manifold identification property. These algorithms include proximal point, proximal gradient, proximal linear, proximal coordinate descent, Douglas-Rachford splitting/alternating direction method of multipliers and iteratively reweighted $\ell_1$ algorithms, covering a broad range of nonsmooth optimization problems. Our work fills a gap in the convergence rate analysis of Anderson-accelerated nonsmooth optimization algorithms and provides a rigorous explanation for the effectiveness of Anderson acceleration when applied to these methods.
\bibliographystyle{abbrv}
\bibliography{refs}
\end{document}